\documentclass{siamart251216}

\usepackage{amsfonts,amssymb}
\usepackage{amsopn}
\usepackage{mathtools}
\usepackage{xcolor}
\usepackage{array}
\usepackage{booktabs}
\usepackage{graphicx}
\usepackage{epstopdf}
\usepackage{float}
\usepackage{subfigure}
\usepackage[algo2e,boxed,ruled,vlined]{algorithm2e}

\DontPrintSemicolon
\SetKwInput{KwRequire}{Require}
\SetKwInput{KwEnsure}{Ensure}
\SetKw{KwRet}{return}
\RestyleAlgo{boxruled}
\SetAlFnt{\small}
\SetAlCapFnt{\small}
\SetAlCapNameFnt{\small}

\newsiamremark{remark}{Remark}
\crefname{remark}{Remark}{Remarks}
\newsiamremark{assumption}{Assumption}
\crefname{assumption}{Assumption}{Assumptions}

\newcommand{\RR}{\mathbb{R}}

\newcommand{\R}{\mathbb{R}}
\newcommand{\SSS}{\mathbb{S}}
\newcommand{\Sn}{\SSS^{n}}
\newcommand{\Snp}{\SSS^{n}_{+}}
\newcommand{\Plin}{\mathcal{P}_{\mathrm{lin}}}

\newcommand{\ip}[2]{\langle #1,\,#2\rangle}
\newcommand{\tr}{\operatorname{tr}}
\newcommand{\diag}{\operatorname{diag}}

\newcommand{\argmax}{\operatorname{argmax}}
\newcommand{\argmin}{\operatorname{argmin}}

\newcommand{\bA}{A}
\newcommand{\bB}{B}
\newcommand{\bC}{C}

\newcommand{\bI}{I}
\newcommand{\bP}{P}
\newcommand{\bQ}{Q}
\newcommand{\bS}{S}
\newcommand{\bstarS}{S^{\star}}
\newcommand{\bU}{U}
\newcommand{\bV}{V}

\newcommand{\X}{X}
\newcommand{\bX}{X}
\newcommand{\bbarX}{\bar{X}}

\newcommand{\bLa}{\Lambda}

\newcommand{\bu}{u}

\newcommand{\bx}{x}
\newcommand{\by}{y}
\newcommand{\bz}{z}

\newcommand{\bq}{q}
\newcommand{\blambda}{\lambda}

\newcommand{\uball}{\mathcal{B}}
\newcommand{\sball}{\Snp\cap\mathcal{B}}

\newcommand{\gcut}{\mathcal{C}}

\newcommand{\norm}[1]{\left\lVert #1\right\rVert}
\newcommand{\normF}[1]{\norm{#1}_{F}}
\newcommand{\normtwo}[1]{\norm{#1}_{2}}

\newcommand{\Xpos}{\bX_{+}}
\newcommand{\Xneg}{\bX_{-}}

\headers{Spectral-gauge cuts for semidefinite programming}{A. Sasaki, S. Demassey, and V. Sessa}

\title{Spectral-gauge cuts for semidefinite programming}
\author{Antonio Sasaki\thanks{CMA, Mines Paris -- PSL, France.}
\and Sophie Demassey\footnotemark[1]
\and Valentina Sessa\footnotemark[1]}

\makeatletter
\@ifundefined{algorithm}
  {\newenvironment{algorithm}[1][]{\begin{algorithm2e}[#1]}{\end{algorithm2e}}}
  {\renewenvironment{algorithm}[1][]{\begin{algorithm2e}[#1]}{\end{algorithm2e}}}
\makeatother
\renewcommand{\mathbf}[1]{#1}
\renewcommand{\boldsymbol}[1]{#1}
\begin{document}
\maketitle
\begin{abstract}
We use symmetric gauge theory to develop a general class of cutting-plane algorithms for semidefinite programming. We formulate a separation problem based on spectral normalizations induced by gauges and derive a closed-form separation oracle. This oracle yields an implementable cut-generation procedure that, by varying the gauge, recovers standard cut families and generates new ones with tunable spectral structure. We embed the oracle within Kelley’s method and characterize convergence as a function of the chosen gauge and initial conic relaxation. Numerical experiments on small and large instances of box-constrained quadratic programming and sparse principal component analysis illustrate the versatility and performance of the proposed framework.
\end{abstract}




\section{Introduction}\label{sec:intro}
We consider semidefinite programs (SDP) in the form~\cite{vandenberghe_boyd_1996}
\begin{equation}
  \label{eq:sdp}
  \begin{aligned}
  \min_{\bX \in \Sn} \quad & \ip{\bC}{\bX} \\
  \text{s.t.}\quad
  & \bX \in \Plin, \\
  & \bX \succeq \mathbf{0},
  \end{aligned}
\end{equation}
where $\bC\in\Sn$ and $\Plin\subseteq\Sn$ is a polyhedron describing the linear side constraints. The space $\Sn$ is equipped with the Frobenius inner product $\ip{\bC}{\bX}\coloneqq \tr(\bC^\top\bX)=\sum_{i,j}C_{ij}X_{ij}$. The notation $\bX\succeq\mathbf{0}$ means that $\bX$ is positive semidefinite (PSD), i.e., $\mathbf{v}^\top\bX\mathbf{v}\ge 0$ for all $\mathbf{v}\in\mathbb{R}^n$. Since $\Plin$ is the intersection of finitely many affine hyperplanes and closed halfspaces in $\Sn$, the feasible set $\mathcal F=\Plin\cap\Snp$ is a spectrahedron~\cite{ramana_goldman_1995,blekherman_parrilo_thomas_2012}, where $\Snp\coloneqq\{\bX\in\Sn : \bX\succeq \mathbf{0}\}$ is the PSD cone.

Interior-point methods are among the most reliable generic algorithms for solving SDPs to high accuracy~\cite{nesterov_nemirovskii_1994,bellavia_gondzio_porcelli_2021}. Their scalability, however, is limited by the dense linear systems that arise at each iteration. In contrast with linear programming, exploiting sparsity in SDPs is often delicate and may require introducing many auxiliary variables~\cite{andersen_etal_2011}. First-order alternatives, including augmented Lagrangian methods~\cite{povh_rendl_wiegele_2006,malick_etal_2009,chen_etal_2025} 
and the alternating direction method of multipliers~\cite{sun_etal_2020,battista_de_santis_2025}, avoid expensive Hessian operations and scale to larger instances. Their main limitation is reduced numerical precision, which can be problematic when tight dual bounds are needed in combinatorial or nonconvex optimization~\cite{daspremont_etal_2007,berk_bertsimas_2019,locatelli_etal_2025,yildirim_2026}.

A widely used alternative is Kelley's cutting-plane approach~\cite{kelley_1960}, which replaces the PSD constraint $\bX\succeq\mathbf{0}$ with a tractable relaxation that is iteratively refined by separating valid inequalities, called \emph{cuts}~\cite{krishnan_mitchell_2006,wang_tanaka_yoshise_2021,bertsimas_cory_wright_2020,sherali_fraticelli_2002}. This approach is particularly useful for convex, nonconvex, and discrete quadratic optimization, where SDP models appear as relaxations of the quadratic relation $\bX= \mathbf{x}\mathbf{x}^\top$~\cite{shor_1987}. In this setting, a linear or conic cutting-plane method that generates a sequence of dual bounds can be kept computationally light by limiting the number of iterations and can be embedded effectively within global-optimization frameworks such as branch-and-bound~\cite{dey_etal_2022}.
A standard cutting-plane framework relies on the semi-infinite linear description
\[
\Snp=\mathcal{P}_1\coloneqq\{\bX\in\Sn : \mathbf{v}^\top \bX\mathbf{v}\ge 0,\ \forall \mathbf{v} \in \mathbb{R}^n,\ \|\mathbf{v}\|_2=1\}.
\]
At a given iterate $\mathbf{\bar{X}}\in\Sn$, separation amounts to finding a negative eigenvalue $\lambda\in\R$, typically the smallest one, together with an associated normalized eigenvector $\mathbf{v}\in\mathbb{R}^n$ such that $\mathbf{v}^\top\mathbf{\bar{X}}\mathbf{v}=\lambda<0$. The resulting valid inequality 
$\ip{\mathbf{v}\mathbf{v}^\top}{X}\ge 0$
is an \emph{eigenvalue cut} (or \emph{eigencut} for short). These cuts, introduced by Ramana~\cite{ramana_1993}, were also considered in~\cite{shor_1998} and in the reformulation-linearization technique~\cite{sherali_fraticelli_2002}, and have since been used regularly in global optimization~\cite{sherali_dalkiran_desai_2012,dey_etal_2022}. Adding many eigenvalue cuts to an initial linear relaxation of~\eqref{eq:sdp} can, however, lead to numerical instability in the master LP. To mitigate this issue, sparse eigenvalue cuts restrict the separation problem to eigenvectors with at most $m$ nonzero entries~\cite{baltean_lugojan_etal_2019,qualizza_belotti_margot_2012,dey_etal_2022}.

Bertsimas and Cory-Wright~\cite{bertsimas_cory_wright_2020} follow a complementary approach by seeking tighter outer approximations in SDP relaxations for machine-learning applications. They relate the empirical performance of cutting-plane methods to the strength of the initial relaxation and use second-order cone relaxations for sparse principal component analysis. They also introduce a deeper class of cuts for nuclear-norm minimization by exploiting the self-duality of the PSD cone
\[
\Snp=\mathcal{P}_\infty\coloneqq \{\bX\in\Sn: \ip{\mathbf{S}}{\bX}\ge 0,\ \forall \mathbf{S}\in\Snp\}.
\]
This identity yields a semi-infinite linear reformulation of~\eqref{eq:sdp} that extends eigenvalue cuts to inequalities of the form
\begin{align}\label{eq:capCS}
\gcut_S:=\ip{\mathbf{S}}{\bX}\ge 0,
\end{align}
with $S$ of arbitrary rank. The eigenvalue inequalities defining $\mathcal P_1$ are the rank-one members of this family. The nuclear cuts of~\cite{bertsimas_cory_wright_2020} can be seen as the most violated inequalities of this type when the separator is normalized in the spectral norm.

\paragraph{Our contributions}
In line with the growing field of convex algebraic geometry~\cite{blekherman_parrilo_thomas_2012}, we use symmetric gauges to give a unified treatment of these cut families and of the continuum between them. Since the inequalities $\gcut_S$ are invariant under positive scaling, the separator $S$ can be restricted to a compact representative slice of $\Snp$. We study the case in which this slice is the unit ball of a unitarily invariant matrix norm. Such norms are induced by symmetric gauges $\phi$, i.e., permutation- and sign-invariant vector norms, through the spectral mapping $\lambda$. Although the eigenvalue map is generally nonsmooth, unitarily invariant norms inherit useful variational and dual properties from their inducing gauges~\cite{daniilidis_etal_2013}.

Our first contribution is to characterize the canonical separation problem over the unit ball of a unitarily invariant norm $\phi\circ\lambda$. We express the maximum violation
in closed form as the image of the negative eigenvalues of the current iterate through the dual gauge $\phi^\circ$, and build the cut from the associated eigenvectors. This yields a practical cut-generation procedure for what we refer to as \emph{spectral-gauge cuts}.

Our second contribution is a convergence analysis for Kelley's method with these spectral-gauge cuts. The analysis uses two geometric quantities: the size of the initial relaxation and the depth of the cuts induced by the chosen gauge. The proof extends the arguments of~\cite{bertsimas_cory_wright_2020} to arbitrary symmetric gauges and initial relaxations. 

The $\ell_p$ norms provide a concrete hierarchy of spectral-gauge cut templates, from
eigenvalue cuts for $\phi=\ell_1$, to nuclear cuts for $\phi=\ell_\infty$, and a continuum 
for $1<p<\infty$, 
where the separation oracle essentially requires
the computation of negative eigenvalues and an associated eigenbasis.

We then add two
restrictions on the separators: 
by enforcing entrywise sparsity, we
connect our framework to 
both sparse eigencuts~\cite{qualizza_belotti_margot_2012,baltean_lugojan_etal_2019,dey_etal_2022} and the factor-width matrix cone hierarchy~\cite{boman_etal_2005,permenter_parrilo_2018};
by enforcing spectral sparsity with a rank limit $k$, we derive a lighter oracle only involving the $k$ most negative eigenvalues.

Numerical experiments on small and large instances of box-constrained quadratic programming and sparse principal component analysis 
illustrate
the practical impact of 
the framework under various configurations.
Overall,
LP relaxation combined with $\ell_1$-eigenvalue cuts or $\ell_2$-Frobenius cuts, with or without rank limit, strikes a good balance between bound quality and computational cost. However, no spectral-gauge cut template is strictly dominated by the others, and each can be selected for its respective strengths, depending on the class and size of a problem, or the iteration number allowed for computing the exact SDP bound or a cheaper dual bound.


\paragraph{Organization of the paper}
Section~\ref{sec:notation} introduces notation and preliminary material on symmetric gauges and unitarily invariant norms. Section~\ref{sec:spectral-gauge} presents the spectral-gauge separation problem and proves the closed-form separation theorem. Section~\ref{sec:instantiation} specializes the theorem to the $\ell_p$ hierarchy, sparse-support cuts, and rank-limit separators. Section~\ref{sec:algorithm} gives the generic cutting-plane algorithm and its convergence analysis. Section~\ref{sec:experiments} reports numerical experiments.

\section{Notation and preliminaries}\label{sec:notation}

Scalars and vectors are denoted by lowercase letters, and matrices by uppercase letters. Let $[n]:= \{1,\dots,n\}$ for short. We use $t_{+}:=\max\{t,0\}$ and $|t|$ for the positive part and absolute value of a scalar $t\in\RR$. For a vector $\bu\in\RR^n$, we define $\bu_+\in\RR_+^n$ and $|\bu|\in\RR_+^n$ componentwise by $(\bu_+)_i=(u_i)_+$ and $|\bu|_i=|u_i|$. The support of $\bu$, denoted $\operatorname{supp}(\bu)$, is the set of indices corresponding to nonzero components. By extension, we define the support of a symmetric matrix $X\in\Sn$ as the set of row indices with nonzero components. We denote by $e_i\in\mathbb{R}^n$ the $i$-th canonical vector. For $\bu\in\R^n$, let $\bu^\uparrow$ and $\bu^\downarrow$ denote the nondecreasing and nonincreasing rearrangements of $\bu$, respectively.

The \emph{spectral mapping} $\boldsymbol{\lambda}$ maps a symmetric matrix to the vector of its eigenvalues, arranged in nondecreasing order and repeated according to multiplicity. For $\bX\in\Sn$, the positive and negative parts of its eigenvalue vector $\lambda(X)=(\lambda_i(X))_i$  by 
$$\boldsymbol{\lambda}_+(\bX) := (\lambda_i(X)_+)_i \in\R_+^n
\mbox{ and } \boldsymbol{\lambda}_- (\bX):= (-\lambda_i(X)_+)_i \in\R_+^n, $$
and the associated diagonal matrices in $\Sn$ are $
\bLa(\bX):=\diag(\lambda(\bX))$, $\bLa_+(\bX) := \diag(\boldsymbol{\lambda}_+(\bX))$, and $\bLa_-(\bX) := \diag(\boldsymbol{\lambda}_-(\bX)).$ 
Throughout the paper, we denote the minimum eigenvalue of the matrix $X$ with both $\lambda_1(X)$ and $\lambda_{\min}(X)$.

Since $\bX$ is symmetric, there exists a real orthogonal matrix $\bQ$, i.e., $\bQ^\top \bQ = \bQ \bQ^\top = \bI_n$, such that $\bX = \bQ \bLa(\bX) \bQ^\top$. We define the \emph{spectral positive part} $\Xpos\in\Sn$ and \emph{spectral negative part} $\Xneg\in\Sn$ as
\begin{equation}\label{eq:spectral-parts}
\Xpos := \bQ\, \bLa_+(\bX) \,\bQ^\top,\quad
\Xneg := \bQ\,\bLa_-(\bX)\,\bQ^\top.
\end{equation}
Then $\Xpos\in\Snp$, $\Xneg\in\Snp$, $\bX=\Xpos-\Xneg$, and $\Xpos \Xneg=\mathbf{0}$.

Since a symmetric matrix $X$ is normal, its singular values, denoted by $\sigma(X)$ in nondecreasing order and counted with multiplicity, satisfy $\sigma(X)=|\lambda(X)|^\uparrow$~\cite[Thm.~2.6.3]{horn_johnson_2013},
and Von Neumann's trace inequality~\cite[Thm.~7.4.1.1]{horn_johnson_2013} reads:
\begin{equation}\label{eq:trineq}
|\ip{\bA}{\bB}|=|\tr(\bA\bB)|
\le \bigl(|\blambda(\bA)|^\uparrow\bigr)^\top\bigl(|\blambda(\bB)|^\uparrow\bigr)
\quad\forall\bA,\bB\in\Sn.
\end{equation}
Furthermore, $X\in\Sn$ is positive semidefinite if and only if all its eigenvalues are nonnegative, equivalently $\lambda_{-}(X)=0$ and $\lambda(X)=|\lambda(X)|^\uparrow=\sigma(X)$~\cite[Thm.~4.1.10]{horn_johnson_2013}.

\begin{definition}[Dual norm~{\cite[Def.~5.4.12]{horn_johnson_2013}}]
\label{def:dualnorm}
Let $\phi$ be a norm on $\mathbb{R}^n$.
The \emph{dual} of $\phi$ is the function $\phi^\circ:\mathbb{R}^n\to\mathbb{R}_+$ defined by
    \begin{equation}\label{eq:dualnorm}
    \phi^\circ(\by)
    :=
    \max_{\bz\in\mathbb{R}^n}
    \bigl\{ |\ip{\by}{\bz}| : \, \phi(\bz)\le 1 \bigr\},
    \quad \forall \by\in\mathbb{R}^n.
    \end{equation}
\end{definition}
The dual $\phi^\circ$ defines a norm on $\R^n$ that satisfies the generalized Cauchy-Schwarz inequality~\cite[Lem.~5.4.13]{horn_johnson_2013}:
\begin{equation}\label{eq:cs}
|\mathbf{y}^\top\mathbf{z}| \le \phi^\circ(\mathbf{y})\,\phi(\mathbf{z})\text{ for all }\by,\bz\in\R^n.
\end{equation}

\begin{definition}[Symmetric gauge {\cite[Def.~7.4.7.1]{horn_johnson_2013}}]
\label{def:norm}
Let $\phi$ be a norm on $\mathbb{R}^n$.
\begin{enumerate}
    \item $\phi$ is \emph{absolute} if $\phi(\bz)=\phi(|\bz|)$ $\forall \bz\in\mathbb{R}^n$.

    \item $\phi$ is a \emph{symmetric gauge} if it is absolute and
    $\phi=\phi\circ P$ for every permutation matrix $\bP\in\{0,1\}^{n\times n}$, that is, a square real matrix with exactly one entry equal to $1$ in each row and column and all other entries equal to $0$.
\end{enumerate}
\end{definition}
From~\cite[Thm.~5.4.19]{horn_johnson_2013}, a norm $\phi$ on $\R^n$ is absolute if and only if it is monotone (i.e., $\phi(\bx)\le \phi(\bz)$ $\forall \bx,\bz\in\mathbb{R}^n$ with $|\bx|\le|\bz|$), and its dual $\phi^\circ$ is also an absolute norm. By monotonicity and sign-invariance, only vectors $z$ with nonnegative entries and support contained in that of $y$ need to be considered in achieving the maximum in~\eqref{eq:dualnorm} when $y$ is nonnegative. Hence, the restriction of $\phi^\circ$ to the nonnegative orthant can be rewritten as
\begin{equation}\label{eq:dual-nonneg}
\phi^\circ(\by)=\max\{\by^\top\bz: \, {\mathbf{z}\in\R^n_+}, \ \phi(\bz)\le 1,  \ \operatorname{supp}(z)\subseteq \operatorname{supp}(y)\}, \quad \forall \by \in\R^n_+.
\end{equation}

The symmetric gauges are also known as \emph{symmetric absolute norms}~\cite[p.~335]{horn_johnson_2013}.
Specializing Von Neumann's result to real symmetric matrices, the next lemma shows that the symmetric gauges on $\R^n$ are in one-to-one correspondence with the 
unitarily invariant matrix norms on $\Sn$, 
i.e., $\norm{\bU\bA\bV}=\norm{\bA}$, $\forall A\in\Sn$, $\forall\,\bU,\bV$ orthogonal,
as well as their respective duals,
through composition with the spectral mapping $\lambda$.

\begin{lemma}[Unitarily invariant matrix norm on $\Sn$]
\label{lem:uin}
Let $\phi$ be a symmetric gauge on $\R^n$.
Define
\[
\norm{S}_{\phi}:=\phi\big(\boldsymbol{\lambda}(S)\big),\quad \forall\,S\in\Sn.
\]
Then $\norm{\cdot}_{\phi}$ is a unitarily invariant matrix norm on $\Sn$.

Conversely, every unitarily invariant matrix norm on $\R^{n\times n}$ restricted to $\Sn$ can be written as $S\mapsto \phi(\boldsymbol{\lambda}(S))$ for a (unique) symmetric gauge $\phi$ on $\R^n$.

The dual matrix norm,
defined by~\eqref{eq:dualnorm} with respect to the Frobenius inner product on $\R^{n\times n}$, is also unitarily invariant and its restriction to $\Sn$ satisfies:
\[
(\phi\circ\lambda)^\circ(X)
=\norm{X}^\circ_{\phi}
=\max_{S\in\Sn}\left\{|\ip{S}{X}\bigr|:\, \norm{S}_{\phi}\le 1\right\} = (\phi^\circ\circ\lambda)(X),
\quad \forall\,X\in\Sn,
\]
where $\phi^\circ$ is the dual gauge of $\phi$ on $\R^n$.
\end{lemma}
\begin{proof}
Mapping $\phi\circ\sigma$ is an unitarily invariant matrix norm~\cite[Thm.~7.4.7.2]{horn_johnson_2013} on $\R^{n\times n}$, then, restricted to $\Sn$,
$\phi(\sigma(S))=\phi(|\lambda(S)|^\uparrow)=\phi(\lambda(S))$ by symmetry of $\phi$.

For the second assertion, we use the result that a norm $\norm{\cdot}$ is unitarily invariant if and only if its dual $\norm{\cdot}^\circ$ is also unitarily invariant~\cite[Thm.~5.6.39]{horn_johnson_2013}.
    By~\eqref{eq:dualnorm}, the dual norm is defined on $\R^{n\times n}$ as
    $\norm{X}_\phi^\circ=\max_{A\in\R^{n\times n}}\left\{|\ip{A}{X}\bigr|:\ \norm{A}_{\phi}\le 1\right\}$
    but it is enough to consider symmetric matrices $A$ when computing this maximum, since
$S=\frac{A+A^\top}{2}\in\Sn$, and then $\ip{S}{X}=\ip{A}{X}$ with
$\norm{S}_{\phi}\le \norm{A}_{\phi}$.

For any feasible $S\in\Sn$ with $\phi(\lambda(S))\le 1$, inequalities~\eqref{eq:trineq} and
\eqref{eq:cs} imply
\[
|\ip{S}{X}|
\le \bigl(|\lambda(S)|^\uparrow\bigr)^\top\bigl(|\lambda(X)|^\uparrow\bigr)
\le \phi\bigl(|\lambda(S)|^\uparrow\bigr)\,\phi^\circ\bigl(|\lambda(X)|^\uparrow\bigr)
\le \phi^\circ(\lambda(X)).
\]
Hence $\|X\|_\phi^\circ\le \phi^\circ(\lambda(X))$.
Conversely, choose a maximizer $z\in\R^n$ of~\eqref{eq:dualnorm} defining $\phi^\circ(\lambda(X))$, i.e., with $\phi(z)\le 1$ and $|z^{\top}\lambda(X)|=\phi^\circ(\lambda(X))$. If $X=Q\diag(\lambda(X))Q^\top$, define $S:=Q\diag(z)Q^\top\in\Sn$. Then $\norm{S}_\phi=\phi(\lambda(S))=\phi(z)\le 1$ by symmetry of $\phi$ and
$
|\langle S,X\rangle|
=|z^{\top}\lambda(X)|
=\phi^\circ(\lambda(X)).
$
Thus $\|X\|_\phi^\circ=\phi^\circ(\lambda(X))$, for all $X \in \Sn$.
\end{proof}

Given a symmetric gauge $\phi$ on $\R^n$, we denote by $\uball_\phi:=\{\bS\in\Sn:\ \norm{\bS}_\phi=\phi(\blambda(\bS))\le 1\}$ the unit ball of the associated matrix norm $\phi\circ\lambda$. This is a compact set~\cite[Cor.~5.4.8]{horn_johnson_2013}, therefore, the dual norm $\norm{\cdot}^\circ_{\phi}$ is well defined. 

\section{Spectral-gauge cuts}\label{sec:spectral-gauge}
This section studies the separation problem for cuts $\gcut_S$ whose separator $S$ lies in a compact representative slice of $\Snp$. We take this slice to be $\Snp\cap\uball_\phi$, where $\uball_\phi$ is the unit ball of the unitarily invariant matrix norm induced by a symmetric gauge $\phi$ on $\R^n$. We give a constructive formula for the most violated cut at a current iterate $X\in\Sn$ and a closed formula for the maximum violation involving the dual gauge $\phi^\circ$ and the negative eigenvalue vector $\lambda_-(X)$.

\begin{definition}[Spectral-gauge separation]\label{def:Bgauge}
Given a symmetric gauge $\phi$ on $\R^n$, the spectral-gauge separation problem at a point $\bbarX\in\Sn$ is the following optimization:
\begin{equation}
  \label{eq:gphi}
  g_\phi(\bbarX) := \min_{\bS\in \Snp\cap\uball_\phi} \ip{\bS}{\bbarX}
  \ \text{ with }\ \uball_\phi=\{\bS\in\Sn: \phi(\lambda(S))\le 1\}.
\end{equation}
If $\bstarS$ solves~\eqref{eq:gphi} and $g_\phi(\bbarX)<0$, then the linear inequality $\gcut_{\bstarS}:=\ip{\bstarS}{\bX}\ge 0$ is called a \emph{spectral-gauge cut} and $S^\star$ an optimal separator.
\end{definition}

Problem~\eqref{eq:gphi} minimizes a continuous function over a compact set; therefore, its minimum $g_\phi(\bbarX)\in\R$ is attained at some $\bstarS\in \Snp \cap\uball_\phi$. By Lemma~\ref{lem:uin}, this value is tightly related to the dual gauge $\phi^\circ$ evaluated at the spectrum of $\bbarX$.

Since every nonzero element $S\in\Snp$ has a representative $S'=S/\norm{S}_\phi$ in $\uball_\phi$, and since $\ip{S}{X}$ and $\ip{S'}{X}$ have the same sign for any $X\in\Sn$, self-duality of the PSD cone can be written as
\begin{equation}\label{eq:dualball}
    \Snp=\mathcal{P}_\phi:=\{\bX\in\Sn: \ip{\mathbf{S}}{\bX}\ge 0,\ \forall \mathbf{S}\in\Snp\cap\uball_\phi\}.
\end{equation}

\begin{remark} Optimizing~\eqref{eq:gphi} over any subset of $\Snp$ leads to a valid inequality $\gcut_{\bstarS}$ that separates $\bar{X}$ if the minimum is negative; optimizing over the representative unit ball ensures completeness: if $g_\phi(\bar{X})\ge 0$ then $\bar{X}\in\Snp$.
\end{remark}

\subsection{Explicit optimal separator}

The following theorem establishes that the problem~\eqref{eq:gphi} is
a complete separation oracle for the semi-infinite linear system~\eqref{eq:dualball}, and it provides a closed formula for $g_\phi$ and an explicit optimal $S^\star$.

\begin{theorem}[Optimal spectral-gauge separator]\label{thm:spectral-gauge}
Let $\phi$ be a symmetric gauge on $\R^n$ and let $\phi^\circ$ be its dual norm. For any $\X\in\Sn$ with eigendecomposition $\bX~=~\bQ\bLa(\bX)\bQ^\top$ and negative spectrum $\blambda_-(\bX)\in\R^n_+$,
the optimal value of the spectral-gauge separation problem~\eqref{eq:gphi} is
\begin{equation}\label{eq:op-seppr}
g_\phi(X)=-\phi^\circ(\blambda_-(X)),
\end{equation}
and there exists an optimal separator
$\bS^\star\in \argmin_{\bS\in\sball_\phi}\ip{\bS}{X}$
of the form
\begin{equation}\label{eq:Sstar-main}
\bS^\star=\bQ\,\diag(\bz^\star)\,\bQ^\top,
\text{ with }
\bz^\star\in\argmax_{\bz\in\R^n_+ }\{\blambda_-(\bX)^\top\bz:\ \phi(\bz)\le 1\}
\end{equation}
such that $\operatorname{supp}(z^\star)\subseteq \operatorname{supp}(\lambda_-(X))$ as in~\eqref{eq:dual-nonneg}. Then, either $g_\phi(X)=0$ and $X\in\Snp$, or $g_\phi(X)<0$ and $\gcut_{\bS^\star}:=\ip{\bstarS}{\bX}\ge 0$
 separates $X$ from $\Snp$.
\end{theorem}

\begin{proof}
  We prove equality~\eqref{eq:op-seppr} by considering $X=\Xpos-\Xneg$ as defined in~\eqref{eq:spectral-parts} by the eigenbasis $Q$ and the positive and negative spectra $\lambda_+(X)$ and $\lambda_-(X)$.

  $(\ge)$: Let $S\in\sball_\phi$. Since $\Xpos\in\Snp$, we have $\ip{S}{\Xpos}\ge 0$ by self-duality, and
  \begin{equation}\label{eq:rel1}
  \ip{S}{X}\ge \ip{S}{-\Xneg}.
  \end{equation}
  Since $\Xneg\in\Snp$, the trace inequality~\eqref{eq:trineq} and the dual equation~\eqref{eq:dual-nonneg} together with $\phi(\lambda(S))\le 1$ yield
\begin{equation}\label{eq:rel2}
  \ip{S}{\Xneg}
  \le \lambda(\Xneg)^\top\lambda(S)
  =\bigl(\lambda_-(X)^\uparrow\bigr)^\top\lambda(S)
  \le \phi^\circ\bigl(\lambda_-(X)^\uparrow\bigr)
  =\phi^\circ(\lambda_-(X)).
\end{equation}
  Minimizing~\eqref{eq:rel1} and maximizing~\eqref{eq:rel2} over $S\in\sball_\phi$ yield $$g_\phi(X)\ge g_\phi(-\Xneg)\ge -\phi^\circ(\lambda_-(X)).$$

$(\le)$: Conversely, let $z^\star\in\R^n_+$ be a solution of the maximization problem defining the dual norm $\phi^\circ(\lambda_-(X))$ in~\eqref{eq:dual-nonneg} with support contained in the support of $\lambda_-(X)$, which means $z^\star$ is orthogonal to $\lambda_+(X)$. Define $S^\star=Q\,\diag(z^\star)\,Q^\top$ using the eigenbasis $Q$ of $X$, then $S^\star\in\sball_\phi$ and
\[
 g_\phi(X) \le
\ip{\bS^\star}{\bX}
=
\ip{\bS^\star}{\Xpos-\Xneg}
=
\blambda_+(\bX)^\top\bz^\star - \blambda_-(\bX)^\top\bz^\star
=
-\phi^\circ(\blambda_-(\bX)).
\]

Consequently, $S^\star$ and $z^\star$ satisfy the conditions in \eqref{eq:Sstar-main}.

The last assertion follows from reformulation~\eqref{eq:dualball}: if $g_\phi(X)\ge 0$ then $g_\phi(X)=0$  as $0\in\sball_{\phi}$, otherwise
$\gcut_{\bstarS}$ is by definition the valid inequality for $\Snp$ with the largest violation $|g_\phi(X)|$ at $X\in\Sn\setminus\Snp$ with respect to $\uball_\phi$.
\end{proof}

Theorem~\ref{thm:spectral-gauge} provides a constructive method for computing a solution $\bstarS$ of the gauge-spectral separation problem~\eqref{eq:gphi}, based on the computation of the dual norm of the negative spectrum of $X$ and an eigenbasis $Q$. Note that, when building $\bstarS$ in~\eqref{eq:Sstar-main}, we can restrict $Q$ to the eigenvectors corresponding to the support of $z^\star$ and thus to the negative eigenvalues of $X$.

\begin{remark}[Spectral-gauges from non-self-dual cones] Although Theorem~\ref{thm:spectral-gauge} is stated for the self-dual cone $\Snp$, the argument behind the separation problem is
more general, 
as it applies to any closed convex cone $\mathcal{K}\subseteq\Sn$, possibly non-self-dual. Indeed, given a symmetric gauge $\phi$ on $\R^n$, the unit ball $\mathcal{B}_\phi$ is compact in $\Sn$, so is the slice $\mathcal{K}\cap \mathcal{B}_\phi$, and $0\in \mathcal{K}\cap \mathcal{B}_\phi$. Moreover, every nonzero $S\in \mathcal{K}$ satisfies $S/\phi(\lambda(S))\in \mathcal{K}\cap \mathcal{B}_\phi$. Hence, the dual cone can be written
\[
\mathcal{K}^*=\{X\in\Sn:\langle S,X\rangle\ge 0, \ \forall S\in \mathcal{K}\cap \mathcal{B}_\phi\},
\]
and the problem
\[
g_{\phi,\mathcal{K}}(X):=\min_{S\in \mathcal{K}\cap \mathcal{B}_\phi}\langle S,X\rangle
\]
is a separation oracle for $\mathcal{K}^*$: one has $g_{\phi,\mathcal{K}}(X)=0$ if and only if $X\in \mathcal{K}^*$, while any minimizer $S^\star\in \mathcal{K}\cap \mathcal{B}_\phi$ with
$g_{\phi,\mathcal{K}}(X)<0$ yields a valid inequality
$\langle S^\star,\cdot\rangle\ge 0$ separating $X$ from $\mathcal{K}^*$.

\end{remark}

\subsection{Cut quality and normalization across gauges}
Cut selection involves several non-equivalent metrics to estimate the strength of an individual cut, or of a family of cuts; see, e.g., \cite{wesselmann_suhl_2012,dey_molinaro_2018}
and~\cite{baltean_lugojan_etal_2019} in the PSD-cut setting. Note first that all cuts $\mathcal{C}_S$ with normalized $S\in\Sn$ are \textit{minimal} for $\Snp$ in the sense that
$$\{X\in\Sn: \langle S',X\rangle\ge 0, \ \forall S'\in \mathcal{K}\}\subseteq \{X\in\Sn: \langle S,X\rangle\ge 0\} \implies S\in \mathcal{K},$$
where $\mathcal{K}$ is any finitely generated cone. 
The separation problem~\eqref{eq:gphi} selects a cut in $\{\mathcal{C}_S:S\in\Sn\}$ with maximum violation $\frac{\langle S,\bar{X}\rangle}{\norm{S}_\phi}$ relative to normalization $\norm{ \cdot }_\phi$ and current iterate $\bar{X}$. This metric measures the potential effectiveness of this individual cut within its family defined by gauge $\phi$, but not the relative strength of spectral-gauge cuts arising from different symmetric gauges.


The \emph{depth-of-cut} is a scale-invariant criterion used for such comparison~\cite{wesselmann_suhl_2012,dey_molinaro_2018}. It measures the 
Euclidean (or Frobenius, in the matrix space) distance between iterate $\bar X$ and the separating hyperplane $H_S := \{X\in\mathbb S^n : \langle S,X\rangle =0\}$, that is
\begin{equation}
\operatorname{depth}(\bar X,S)
    := \operatorname{dist}_F(\bar X,H_S)
    =
    \frac{|\langle S,\bar X\rangle|}{\|S\|_F}.
       \label{eq:cut-depth}
\end{equation}
The above measure can be easily derived considering that the Frobenius projection of $\bar X$ onto $H_S$ is
$
    \operatorname{proj}_{H_S}(\bar X)
    =
    \bar X
    -
    \frac{\langle S,\bar X\rangle}{\|S\|_F^2}S.
$
The depth has a simple interpretation in our spectral
setting. Let $\bar X=\bar X_+-\bar X_-$ be the decomposition as in \eqref{eq:spectral-parts}, then using Cauchy--Schwarz, we have for every
$S\in\Snp$
\[
    -\frac{\langle S,\bar X\rangle}{\|S\|_F}
    =
    \frac{\langle S,\bar X_-\rangle-\langle S,\bar X_+\rangle}{\|S\|_F}
    \le
    \frac{\langle S,\bar X_-\rangle}{\|S\|_F}
    \le
    \|\bar X_-\|_F.
\]
Equality is attained by any positive multiple of $\bar X_-$, and, in particular, by the optimal spectral-gauge separator for the symmetric gauge $\ell_2$ (see Corollary~\ref{cor:frob}).  

A third criterion is the complexity of the separator. The rank and the sparsity of the separator have a direct impact on the computational cost of the LP reoptimization that follows each oracle call. According to Theorem~\ref{thm:spectral-gauge}, one may choose an optimal separator $S^\star$ whose rank is at most the number of negative eigenvalues of the current iterate $\bar X$, denoted $p$. Hence, after possibly relabeling the negative eigenvectors, $S^\star$ admits the spectral decomposition
$S^\star=\sum_{i\in J} \alpha_i q_i q_i^\top, \:\alpha_i>0,
$
for some $J\subseteq [p]$, where the vectors $q_i$ are eigenvectors associated with negative eigenvalues of $\bar X$.

Geometrically, the separator reflects the spectral structure of the iterate: the negative eigenspace identifies directions along which the quadratic form $q_i^\top \bar X q_i$ is negative, and each eigenvector $q_i$ associated with a negative eigenvalue corresponds to a principal direction of curvature violating the semidefinite constraint. The rank-one matrices $q_i q_i^\top$ therefore define elementary separating hyperplanes in the space of symmetric matrices, each enforcing the inequality $q_i^\top X q_i \ge 0$. When all $p$ directions are used, the $p$-rank separator aggregates all such violated spectral directions, producing a cut whose normal vector spans the entire negative eigenspace. Hence, such a full-rank cut reshapes the feasible set more drastically than a rank-one cut and could move the next optimum to a distant region, therefore increasing the cost of reoptimization. Selecting only a subset of eigenvectors yields a separator whose normal lies in a lower-dimensional subspace of this eigenspace.

A distinct approach to reducing computational cost is to sparsify the separator. Note that rank and entrywise sparsity are not related. Indeed,
rank-one separators are typically dense because the eigenvectors $q_i$ themselves are dense, so the resulting cut $\mathcal{C}_{S^\star}$ is not significantly sparser.
This distinction motivates the sparse and low-rank optimal separators developed in Sections~\ref{subsec:sparse_support_gauges} and~\ref{subsec:sparsified_gauges}, respectively.

\section{Instantiation}\label{sec:instantiation}
This section illustrates the canonical framework. Section~\ref{subsec:lp} applies Theorem~\ref{thm:spectral-gauge} to the $\ell_p$ norms, which recover standard eigencuts for $p=1$, nuclear cuts for $p=\infty$, and infinitely many intermediate cuts. Section~\ref{subsec:sparse_support_gauges} connects our framework to the 
factor-width
hierarchy~\cite{permenter_parrilo_2018} and to sparse eigencuts~\cite{dey_etal_2022}. Section~\ref{subsec:sparsified_gauges} characterizes optimal separators under a prescribed rank limit.

\subsection{Spectral-$\ell_p$ cut hierarchy}\label{subsec:lp}
The $\ell_p$ norms form a well-known family of symmetric gauges on $\R^n$ that are closed under duality in the following sense:
if $\phi$ is the $\ell_p$ norm for some $p\in[1,+\infty]$, i.e., $\phi(\bz)=\|\bz\|_p=(\sum |z_i|^p)^{1/p}$ or $\phi(\bz)=\|\bz\|_\infty=\max |z_i|$, then the dual norm $\phi^\circ$ is the $\ell_q$ norm where $1/p+1/q=1$ with the convention $1/\infty=0$ (see, e.g.,~\cite{horn_johnson_2013}). The pair $(p,q)$ is called a pair of \emph{dual exponents}.
Theorem~\ref{thm:spectral-gauge} applies to the $\ell_p$ norms as follows.

\begin{proposition}[Closed-forms for $\ell_p$ gauges]\label{prop:lp}
Let $\phi=\norm{\cdot}_p$ be the $\ell_p$-norm for some dual exponents $p,q\in[1,\infty]$. Then, for any $\X\in\Sn$ with eigendecomposition $\bX =\bQ\bLa(\bX)\bQ^\top$ and  spectrum $\blambda(\bX)=(\lambda_i)_i$, the optimum value of the spectral-gauge separation problem~\eqref{eq:gphi} is
\[
g_\phi(\bX) = -\|\blambda_-(\bX)\|_q,
\]
with optimal solution
$\bS^\star=\bQ\diag(\bz^\star)\bQ^\top$
 for $\bz^\star\in\argmax_{z\in\R^n_+}\{\blambda_-(\bX)^\top \bz:\|\bz\|_p\le 1\}.$

In particular, if $1<p<\infty$ and $\bX\in\Sn\setminus\Snp$, then a maximizer is defined componentwise by
$$z^\star_i = \left(\frac{\lambda_-(X)_i}{\norm{\lambda_-(X)}_q}\right)^{q-1}\quad\forall i\in [n].$$

\end{proposition}

\begin{proof}

The first part is Theorem~\ref{thm:spectral-gauge} applied to $\phi=\|\cdot\|_p$ and $\phi^\circ=\|\cdot\|_q$.
It remains to compute the maximizer $\bz^\star$ when $1<p<\infty$ and $\bX\not\in\Snp$, i.e., $\blambda_-(\bX)\neq\mathbf{0}$.

Let $y=\lambda_-(X)\in\R_+^n$. The Cauchy-Schwarz inequality~\eqref{eq:cs} reads $y^\top\bz^\star \le \|y\|_q\|\bz^\star\|_p$
with equality if and only if there exists $\alpha>0$ such that
$(z^\star_i)^p=\alpha\,(y_i)^q$ for all $i\in [n]$~\cite[p.~559]{horn_johnson_2013}. Maximum occurs for $\|\bz^\star\|_p=1$, which yields
\[
1=\|\bz^\star\|_p^p
 = \sum_{i=1}^n (z_i^\star)^p
 = \alpha \sum_{i=1}^n y_i^q
 = \alpha \norm{y}_q^q.
\]

Hence, $z_i^\star=\alpha^{1/p}(y_i)^{q/p}
= \norm{y}_q^{-q/p}\,(y_i)^{q/p}$, with $q/p=q-1$.
\end{proof}

We now give closed forms for the separation problem corresponding to the Manhattan norm $\ell_1$, the Euclidean $\ell_2$ norm, and the maximum norm $\ell_\infty$, recognizing the standard eigenvalue cuts in the former case.

\begin{corollary}[Eigencuts ($p=1$)]\label{cor:rank1}
$\Snp\cap\uball_{\ell_1}=\{\bS\in\Snp:\tr(\bS)\le 1\}$ and
\[
g_{\ell_1}(\bX)=-\|\blambda_-(\bX)\|_\infty=\min\{\lambda_1(\bX),0\}.
\]
Moreover, if $\lambda_1(\bX)<0$, an optimal separator may be chosen as
$\bS^\star=\bq_1\bq_1^\top$, where $\bq_1$ is a unit eigenvector associated with the smallest eigenvalue $\lambda_1(\bX)$.
\end{corollary}

\begin{proof}
For $\bS\in\Snp$, $\|\blambda(\bS)\|_1=\tr(\bS)$. Apply Proposition~\ref{prop:lp} with $p=1$ and dual exponent $q=\infty$.
If $\lambda_1(\bX)<0$, then
$\bS^\star=\bq_1\bq_1^\top\in\Snp$ with $\tr(\bS^\star)=1$, and
$
\ip{\bS^\star}{\bX}
=\lambda_1(\bX)=g_{\ell_1}(\bX),
$
so $\bS^\star$ is optimal.
\end{proof}

\begin{corollary}[Frobenius cuts ($p=2$)]\label{cor:frob}
$\uball_{\ell_2}=\{\bS\in\Sn:\normF{\bS}\le 1\}$ and
$g_{\ell_2}(\bX)=-\|\Xneg\|_F$. If $\Xneg\neq\mathbf{0}$, an optimal separator is given by
\[
\bS^\star = \frac{\Xneg}{\|\Xneg\|_F}.
\]
\end{corollary}

\begin{proof}
Apply Proposition~\ref{prop:lp} with $p=q=2$ and $\|\Xneg\|_F=\|\blambda_-{(\bX)}\|_2$.
\end{proof}


\begin{corollary}[Nuclear cuts ($p=\infty$)]\label{cor:projector} $\uball_{\ell_\infty}=\{\bS\in\Sn:\normtwo{\bS}\le 1\}$ and
\[
g_{\ell_\infty}(\bX)=-\|\blambda_-(\bX)\|_1=\sum_{\lambda_i(\bX)<0}\lambda_i(\bX).
\]
An optimal separator may be chosen as the (unique) orthogonal projector 
onto the negative eigenspace of $\bX$,
i.e., $\bS^\star=
Q\,\operatorname{diag}\bigl(\mathbf 1_{\{\lambda_i(X)<0\}}\bigr)\,Q^\top$.
\end{corollary}

\begin{proof}
Apply Proposition~\ref{prop:lp} with $p=\infty$ and $q=1$. The vector maximization reads
$z^\star\in\argmax\{\blambda_-(X)^\top\bz:\ \bz\ge\mathbf{0},\ \|\bz\|_\infty\le 1\}$ and holds for $z^\star_i=1$ if $\lambda_-(X)_i>0$ and $z^\star_i=0$ otherwise.
\end{proof}
The latter case $p=
\infty$ coincides with the definition of nuclear cut in~\cite{bertsimas_cory_wright_2020}:
\[
\langle W,X\rangle \le \operatorname{tr}(X),
\quad W \in \Snp,\ \|W\|_2 \le 1.
\]
Let $S := {I-W}$. Since $W \in \Snp$ and $\|W\|_2 \le 1$, its eigenvalues lie in $[0,1]$. Hence,
\begin{equation*}
\max_{\substack{W\in \Snp\\ \|W\|_2\le 1}}
\bigl(\langle W,X\rangle-\operatorname{tr}(X)\bigr) =
\max_{\substack{W\in \Snp\\ \|W\|_2\le 1}}
\langle W-I,X\rangle =
-\min_{\substack{S\succeq 0\\ \|S\|_2\le 1}}
\langle S,X\rangle = -g_{\ell_\infty}(X).
\end{equation*}


\begin{remark}[Relative strengths]
The spectral-$\ell_p$ cuts can be compared in terms of the amount of violation $-g_{\ell_p}$ at a given iterate $X$. For $1\le p\le q\le \infty$, $\norm{\bz}_{q}\le \norm{\bz}_{p}$ holds $\forall\bz\in\RR^n$. Hence, for any $\bS\in \uball_{\ell_p}$, one has $\norm{\blambda(\bS)}_{q}\le \norm{\blambda(\bS)}_{p}\le 1$, proving $\uball_{\ell_p}\subseteq \uball_{\ell_q}$ and $g_{\ell_q}(\bX)\le g_{\ell_p}(\bX)$.
Alternatively, under the Euclidean depth measure~\eqref{eq:cut-depth},
Frobenius cuts are the deepest. Hence, spectral-$\ell_\infty$ cuts may yield larger violation, but spectral-$\ell_2$ cuts maximize depth.
\end{remark} 

\subsection{Spectral-gauge cuts with sparse support}
\label{subsec:sparse_support_gauges}
The high density of the eigencuts, i.e., the $\ell_1$-spectral gauge cuts, is an issue addressed in~\cite{qualizza_belotti_margot_2012,baltean_lugojan_etal_2019,dey_etal_2022} by enforcing the separating eigenvector ($q_1$ in Corollary~\ref{cor:rank1}) to have small
support.
We extend this idea 
to spectral-gauge cuts for arbitrary symmetric gauge $\phi$. 
A direct analogue is to require the separator $S$ to be supported on a block $I \times I$, so that the inequality $\langle S,X\rangle \geq 0$
only involves the principal submatrix $X_I$. This restriction is closely related
to the factor-width matrix cone hierarchy~
\cite{boman_etal_2005,permenter_parrilo_2018} 
the conic hull of PSD principal blocks of size at most $m$ is the factor-width-$m$
cone, and its dual is the cone of matrices whose principal submatrices of order at
most $m$ are all PSD.

We treat this additional sparsity requirement in two steps. First, we fix a support
$I \subseteq [n]$ and show that the sparse spectral-gauge oracle reduces 
to
Theorem~\ref{thm:spectral-gauge} applied to the principal submatrix $X_I$. Second, we fix
a maximum support size $m$ and optimize over all supports $I$ with $|I|\le m$. This
separates the spectral part of the oracle, which remains explicit, from the outer
combinatorial search over supports. In the special case $\phi=\ell_1$, this
recovers the sparse eigencut separation problem of~\cite{dey_etal_2022}.

\paragraph{Fixed support $I$}
Fix $I\subseteq[n]$ and let
$\mathbf{E}_I\in\{0,1\}^{n\times |I|}$ be the principal block selector, i.e., for any $\mathbf{X}\in\mathbb{S}^n$, the product
$\mathbf{E}_I^\top \mathbf{X}\mathbf{E}_I$ is the principal submatrix of
$\mathbf{X}$ supported on $I$, hence an element of $\mathbb{S}^{|I|}$. We denote
\[
\mathbf{X}_I
:=
\mathbf{E}_I^\top \mathbf{X}\mathbf{E}_I
\in\mathbb{S}^{|I|},
\quad
\iota_I(\mathbf{T})
:=
\mathbf{E}_I\mathbf{T}\mathbf{E}_I^\top
\in\mathbb{S}^n
\]
the principal-block projection and the corresponding lifting map, respectively. A matrix
$\mathbf{S}\in\mathbb{S}^n$ is supported on the principal block $\operatorname{supp}(\mathbf{S})\subseteq I$ if $\mathbf{S}=\iota_I(\mathbf{S}_I)$ for some
$\mathbf{S}_I\in\mathbb{S}^{|I|}$. 
 The block PSD cone associated with $I$ and its dual cone in $\Sn$ are
\[
\iota_I(\mathbb{S}_+^{|I|})
=
\Big\{
\mathbf{E}_I\mathbf{S}_I\mathbf{E}_I^\top:
\mathbf{S}_I\in\mathbb{S}_+^{|I|}
\Big\}
\subseteq \Snp,\quad
\iota_I(\mathbb{S}_+^{|I|})^*=
\{\mathbf{X}\in\mathbb{S}^n:\ \mathbf{X}_I\in\mathbb{S}^{|I|}_+\}.
\]
since $\big\langle
\mathbf{E}_I\mathbf{S}_I\mathbf{E}_I^\top,\mathbf{X}\big\rangle
=
\big\langle
\mathbf{S}_I,\mathbf{X}_I
\big\rangle$.

Given a symmetric gauge $\phi$ on $\mathbb{R}^n$, the support $I$ induces a symmetric gauge
$\phi_I$ on $\mathbb{R}^{|I|}$ by zero-padding
$\phi_I(\mathbf{u})
:=
\phi((\mathbf{u},0_{[n]\setminus I}))$ for $
\mathbf{u}\in\mathbb{R}^{|I|}$,
where the positions of the zeros are not relevant due to the permutation invariance of
$\phi$. Equivalently, its dual gauge
$\phi_I^\circ$ is the restriction of $\phi^\circ$ to vectors supported on $I$. 
Enforcing a fixed-support condition restricts the spectral-gauge separation problem to
\begin{equation}\label{eq:full-space}
g_{I,\phi}(\bar{\mathbf{X}})
:=
\min
\Big\{
\langle \mathbf{S},\bar{\mathbf{X}}\rangle:
\mathbf{S}\in\Snp\cap\uball_\phi,\ 
\operatorname{supp}(\mathbf{S})\subseteq I
\Big\}.
\end{equation}

Since $\lambda(\iota_I(\mathbf{S}_I))
=
\big(\lambda(\mathbf{S}_I),0_{[n]\setminus I}\big)$, the feasible set is the lifting of the lower-dimensional ball:
\[
\Snp\cap\uball_\phi\cap
\{\mathbf{S}:\operatorname{supp}(\mathbf{S})\subseteq I\}
=
\Big\{
\mathbf{E}_I\mathbf{S}_I\mathbf{E}_I^\top:
\mathbf{S}_I\in\mathbb{S}_+^{|I|},\
\phi_I\big(\boldsymbol{\lambda}(\mathbf{S}_I)\big)\le 1
\Big\} =:\uball_{I,\phi}
\]
Consequently, the restricted problem~\eqref{eq:full-space} projects onto
the subspace $\mathbb{S}^{|I|}$,
resulting in the spectral-gauge oracle of Theorem~\ref{thm:spectral-gauge} applied to the
principal block $\bar{X}_I\in\mathbb{S}^{|I|}$ of the current iterate $\bar{X}$. These results are formalized in the following proposition.

\begin{proposition}[Fixed-support spectral-gauge oracle]
\label{prop:fixed_support_spectral_gauge}
Given $\bar{X}\in\mathbb{S}^n$,
\[
g_{I,\phi}(\bar{X})
=
\min_{\substack{\mathbf{S}\in\mathbb{S}^{|I|}_+\\
\phi_I (\boldsymbol{\lambda}(\mathbf{S}))\le 1}}
\langle \mathbf{S},\bar{X}_I\rangle
=
-\phi_I ^\circ\!\big(\boldsymbol{\lambda}_-(\bar{X}_I)\big).
\]
Thus, $g_{I,\phi}(\bar{X})=0
\iff
\bar{X}\in \iota_I(\mathbb{S}_+^{|I|})^*$. If $g_{I,\phi}(\bar{X})<0$ and $\bar{X}_I
=
\mathbf{Q}\,\operatorname{diag}\!\big(\boldsymbol{\lambda}(\bar{X}_I)\big)\,\mathbf{Q}^\top$,
\[
z^\star\in
\argmax\left\{
\lambda_-(\bar X_I)^\top z:
z\in\mathbb R^{|I|}_+,\ 
\phi_I(z)\le 1,\ 
\operatorname{supp}(z)\subseteq
\operatorname{supp}(\lambda_-(\bar X_I))
\right\},
\]
then an optimal separator for $g_{I,\phi}(\bar{X})$ is
\[
\mathbf{S}_I^\star
=
\mathbf{E}_I\,
\mathbf{Q}\,\operatorname{diag}(\mathbf{z}^\star)\,\mathbf{Q}^\top
\mathbf{E}_I^\top
\in \mathcal{B}_{I,\phi}\subseteq \iota_I(\mathbb{S}_+^{|I|}).
\]
Moreover, $\langle \mathbf{S}_I^\star,\mathbf{X}\rangle\ge 0$ for all
$\mathbf{X}\in \iota_I(\mathbb{S}_+^{|I|})^*$, and in particular for every $\mathbf{X}\in\Snp$.
\end{proposition}


\paragraph{Variable support of bounded size $m$}
We now fix a maximum support size and allow the support to vary. Following the factor-width terminology of
Boman et al.~\cite{boman_etal_2005} and the cone
formulation used by Permenter and Parrilo~\cite{permenter_parrilo_2018},
we define for $1\le m\le n$ the factor-width-$m$ cone by
\[
\mathcal{FW}_m^n
:=
\Big\{
\sum_{\substack{I\subseteq[n]\\ |I|\le m}}
\iota_I(S_I):
S_I\in\mathbb S^{|I|}_+
\Big\} =
\operatorname{cone}\Big(
\bigcup_{\substack{I\subseteq[n]\\ |I|\le m}}
\iota_I(\mathbb S^{|I|}_+)
\Big),
\]
where the sum above is the Minkowski sum of cones. This formulation is
equivalent to Boman’s definition of factor-width, i.e., every
term $\iota_I(S_I)$, with $S_I\succeq0$, is a sum of rank-one
matrices $vv^\top$ with $|\operatorname{supp}(v)|\le m$, and
conversely every such rank-one term is supported on a principal block of
size at most $m$~\cite[Sec.~5.3]{ahmadi_majumdar_2019}.

By self-duality of the PSD cone, $\iota_I(\mathbb S^{|I|}_+)^*
=
\{X\in\mathbb S^n:\ X_I\succeq0\}$. Hence
\begin{equation}\label{eq:m-local}
(\mathcal{FW}_m^n)^*
=
\bigcap_{\substack{I\subseteq[n]\\ |I|\le m}}
\iota_I(\mathbb S^{|I|}_+)^*
=
\left\{
X\in\mathbb S^n:
X_I\succeq0\quad
\forall I\subseteq[n],\ |I|\le m
\right\} =: \mathbb{S}^{n,m}_+.
\end{equation}


Define
\[
\mathcal{A}_{m,\phi}
:=
\bigcup_{\substack{I\subseteq[n]\\ |I|\le m}}
\mathcal{B}_{I,\phi},
\quad
g_{m,\phi}(\mathbf{X})
:=
\min_{\mathbf{S}\in \mathcal{A}_{m,\phi}}
\langle \mathbf{S},\mathbf{X}\rangle.
\]
For every $\bar{X}\in\mathbb{S}^n$, the sparse spectral-gauge separation problem is
\[
g_{m,\phi}(\bar{X})
=
\min_{\substack{I\subseteq[n]\\ |I|\le m}}
g_{I,\phi}(\bar{X})
=
-\max_{\substack{I\subseteq[n]\\ |I|\le m}}
\phi_I ^\circ\!\big(\boldsymbol{\lambda}_-(\bar{X}_I)\big).
\]
Moreover, $g_{m,\phi}(\bar{X})=0
\iff
\bar{X}\in \mathbb{S}^{n,m}_+$. If the maximum is attained with support
$I^\star$, then the optimizer
$\mathbf{S}_{I^\star}^\star$ from Proposition~\ref{prop:fixed_support_spectral_gauge}
is optimal for $g_{m,\phi}(\bar{X})$. If
$g_{m,\phi}(\bar{X})<0$, then
$\langle \mathbf{S}_{I^\star}^\star,\mathbf{X}\rangle\ge 0$ for all
$\mathbf{X}\in\mathbb{S}^{n,m}_+$, and in particular for every $\mathbf{X}\in\Snp$. 

For $\phi=\ell_1$, this reduces to the
sparse eigencut separation~\cite{dey_molinaro_2018}:
\begin{equation}\label{eq:sparse-eigencut}
g_{m,\ell_1}(\mathbf{X})
=
-\max_{\substack{I\subseteq[n]\\ |I|= m}}
\big(-\lambda_{\min}(\mathbf{X}_I)\big)_+.
\end{equation}

This decomposition shows where the complexity of the oracle lies: given a prescribed support $I$, separation is the
spectral-gauge problem restricted to the block $\bar{X}_I$. The nontrivial part
is the outer maximization over $I$, which ranges over 
$\binom{n}{m}$ candidate supports and is thus combinatorial. 

\paragraph{General sparsity patterns}
Following Günlük et al.~\cite{gunluk_etal_2026},
let $E \subseteq [n]\times[n]$ be a symmetric index set containing the diagonal,
and define
\[
H_E := \{ \mathbf{S} \in \mathbb{S}^n : S_{ij} = 0 \ \forall (i,j)\notin E \},
\quad
\mathcal{K}_E := \mathbb{S}^n_+ \cap H_E.
\]
Given a symmetric gauge $\phi$ on $\mathbb{R}^n$, consider
\[
g_{E,\phi}(\mathbf{X})
:=
\min \Big\{
\langle \mathbf{S}, \mathbf{X} \rangle:\
\mathbf{S} \in \mathcal{K}_E \cap \uball_\phi
\Big\}.
\]
The condition $g_{E,\phi}(\mathbf{X}) \ge 0$ is equivalent to nonnegativity against
all $E$-supported PSD separators, and thus yields a reformulation of the SDP-$E$ relaxation~\cite[Eq.~(5)]{gunluk_etal_2026}. In particular,
when $\phi = \ell_1$, one recovers the problem of
\cite[Eq.~(15)]{gunluk_etal_2026}. 

For a general $E$, however, $\mathcal{K}_E$ is not necessarily orthogonally invariant, so the diagonalization argument of Theorem~\ref{thm:spectral-gauge} no longer implies a closed form. Such a formula is available only when the cone $\mathcal{K}_E$ is compatible with the spectral structure, as in the full PSD cone or in a fixed principal-block restriction.

\subsection{Spectral-gauge separators with rank limit}
\label{subsec:sparsified_gauges}
As discussed, sparsifying the separator can significantly complicate the separation problem. Instead, we propose to 
control the rank of
the separator, and show that it preserves tractability while still controlling spectral complexity.

\begin{lemma}[Rank and spectral $\ell_0$-function]
\label{lem:rank_equals_l0_spectrum}
Let $\mathbf{S}\in\mathbb{S}^n$ with spectral mapping $\boldsymbol{\lambda}(\mathbf{S})\in\mathbb{R}^n$. Then $\operatorname{rank}(\mathbf{S})=\|\boldsymbol{\lambda}(\mathbf{S})\|_0$, where $\|\mathbf{x}\|_0:=\big|\{i\in[n]:\, x_i\neq 0\}\big|$.
\end{lemma}
\begin{proof}
Decompose $\mathbf{S}=\mathbf{Q}\operatorname{diag}(\boldsymbol{\lambda}(\mathbf{S}))\mathbf{Q}^\top$ with $\mathbf{Q}$ orthogonal. Since $\mathbf{Q}$ is invertible, $\operatorname{rank}(\mathbf{S})=\operatorname{rank}\!\big(\operatorname{diag}(\boldsymbol{\lambda}(\mathbf{S}))\big)$, which is the number of nonzero eigenvalues.
\end{proof}

\begin{proposition}\label{prop:sparsified_basic}
Let $\phi$ be a symmetric gauge on $\mathbb{R}^n$, let $k\in[n]$, and define
\begin{equation*}
\mathcal{A}_{k}(\phi)
:=
\Big\{\mathbf{u}\in\mathbb{R}^n:\ \phi(\mathbf{u})\le 1,\ \|\mathbf{u}\|_0\le k\Big\},
\quad
\mathcal{K}_{k}(\phi):=\operatorname{conv}\big(\mathcal{A}_{k}(\phi)\big),
\end{equation*}
together with the associated \emph{gauge function} \cite[Eq.~(2)]{chandrasekaran_etal_2012}
\begin{equation}
\phi^{\langle k\rangle}(\mathbf{z})
:=
\inf\Big\{t>0:\ \mathbf{z}\in t\,\mathcal{K}_{k}(\phi)\Big\}\quad\forall z\in\R^n.
\label{eq:phi_m_def}
\end{equation}
Then $\phi^{\langle k\rangle}$ is a symmetric gauge on $\mathbb{R}^n$ with unit ball $\mathcal{K}_{k}(\phi)$.
\end{proposition}

The proof of Proposition~\ref{prop:sparsified_basic} is given in Appendix~\ref{app:proof_phi}. For $y\in\R^n$, let $I_k(\mathbf{y}) \subseteq [n]$ be the index set of $k>0$ largest entries of $|\mathbf{y}|$, and
\[
\operatorname{trun}_k(\mathbf{y})
:=
\sum_{i \in I_k(\mathbf{y})} y_i \mathbf{e}_i \in \R^n.
\]

\begin{proposition}
\label{prop:sparsified_dual}
Let $\phi$ be a symmetric gauge and define $\phi^{\langle k\rangle}$ by \eqref{eq:phi_m_def}, then 
\begin{align*}
\big(\phi^{\langle k\rangle}\big)^\circ(\mathbf{y}) 
&=
\max\Big\{\mathbf{y}^\top\mathbf{u}:\ \phi(\mathbf{u})\le 1,\ \|\mathbf{u}\|_0\le k, u\in\R^n\Big\}\\
&=
\max_{\substack{I\subseteq[n]\\ |I|\le k}}\ \phi^\circ(\mathbf{y}\odot \mathbf{1}_I) =
\phi^\circ\big(\operatorname{trun}_k(\mathbf{y})\big)\quad\forall\mathbf{y}\in\mathbb{R}^n,
\end{align*}
where $\odot$ denotes componentwise multiplication and $\mathbf{1}_I$ is the indicator vector of $I$.
\end{proposition}

\begin{proof}
The linear function $\langle y, \cdot \rangle$ attains its maximum on the unit ball $\mathcal{K}_k(\phi)=\operatorname{conv}(\mathcal{A}_k(\phi))$ in $\mathcal{A}_k(\phi)$, giving the first equality. For fixed $I$, maximizing over vectors supported on $I$ gives $\phi^\circ(\mathbf{y}\odot\mathbf{1}_I)$. By monotonicity and symmetry of $\phi^\circ$, the maximum over all $|I|\le k$ is attained by keeping the $k$ largest magnitudes of $\mathbf{y}$. 
\end{proof}

\begin{theorem}
\label{cor:sparsified_oracle}
Let $\phi$ be any symmetric gauge on $\R^n$ and $\phi^{\langle k\rangle}$ defined in~\eqref{eq:phi_m_def}. Then, for $\mathbf{X}\in\mathbb{S}^n$ with spectral decomposition $\mathbf{X}=\mathbf{Q}\,\Lambda(\mathbf{X})\,\mathbf{Q}^\top$, the spectral gauge separation of $\mathbf{X}$ from $\Snp$ reads

\[
g_{\phi^{\langle k\rangle}}(\mathbf{X})
=
-\big(\phi^{\langle k\rangle}\big)^\circ(\lambda_-(\mathbf{X}))
=
-\phi^\circ\big(\operatorname{trun}_k(\lambda_-(\mathbf{X}))\big).
\]
Moreover, there exists an optimal separator of the form $\mathbf{S}^\star
=
\mathbf{Q}\,\operatorname{diag}(\mathbf{z}^\star)\,\mathbf{Q}^\top$ with
\[
\mathbf{z}^\star \in \argmax\{(\operatorname{trun}_k(\lambda_-(\mathbf{X})))^\top \mathbf{z}:\ \mathbf{z}\ge \mathbf{0},\ \phi(\mathbf{z})\le 1, \ \operatorname{supp}(z)\subseteq \operatorname{supp}\!\left(\operatorname{trun}_k(\lambda_-(X))\right)\}.
\]
In particular, $\operatorname{rank}(\mathbf{S}^\star)\le k$. If $\mathbf{X}\not\succeq \mathbf{0}$, then $\langle \mathbf{S}^\star,\mathbf{X}\rangle<0$ and the cut $\mathcal{C}_{S^\star}$ is valid.
\end{theorem}

\begin{proof}
The formula follows by the application of Theorem~\ref{thm:spectral-gauge} to $\phi^{\langle k\rangle}$, and then  Proposition~\ref{prop:sparsified_dual}. If $\mathbf{z}^\star$ solves the vector problem, we may replace it by $\mathbf{z}^\star\odot\mathbf{1}_I$, where $I\subseteq[n]$ is the support of  $\operatorname{trun}_k(\lambda_-(\mathbf{X}))$. Then $\mathbf{z}^\star\in\mathcal{A}_k(\phi)$, so $\phi^{\langle k\rangle}(\mathbf{z}^\star)\le 1$, and thus $\mathbf{S}^\star=\mathbf{Q}\operatorname{diag}(\mathbf{z}^\star)\mathbf{Q}^\top\in \sball_{\phi^{\langle k\rangle}}$ with $\operatorname{rank}(\mathbf{S}^\star)=\|\mathbf{z}^\star\|_0\le k$ by Lemma~\ref{lem:rank_equals_l0_spectrum}. Finally, $\langle \mathbf{S}^\star,\mathbf{X}\rangle
=
-(\operatorname{trun}_k(\lambda_-(\mathbf{X})))^\top\mathbf{z}^\star
=
-\phi^\circ(\operatorname{trun}_k(\lambda_-(\mathbf{X})))$, so $\mathbf{S}^\star$ is optimal. 
\end{proof}

Thus, the rank-limit parameterization boils down to evaluating the original dual gauge on the $k$ most negative eigenvalues. In particular, the $\ell_1$ eigencuts result from the rank restriction $k=1$ applied to any $\ell_p$-gauge, $1\le p\le\infty$.


\begin{remark}[Dual $k$-norm]
Specializing Proposition~\ref{prop:sparsified_dual} and Theorem~\ref{cor:sparsified_oracle} to $\phi=\ell_\infty$ for any $k\in[n]$ gives
\[
(\ell_\infty^{\langle k\rangle})^\circ(y)
=
\|\operatorname{trun}_k(y)\|_1
=
\sum_{i=1}^k |y|_{(i)}
=: \|y\|_{(k)}\quad\text{and}\quad g_{\ell_\infty^{\langle k\rangle}}(X)
=
-\|\lambda_-(X)\|_{(k)}
\]
where $\|\cdot\|_{(k)}$ denotes the vector $k$-norm, the symmetric gauge defined as the sum of the
$k$ largest absolute components. 
Therefore, 
these $k$-rank nuclear cuts coincide with the $\|\cdot\|_{(k)} ^\circ$-spectral gauge cuts.
\end{remark}

\section{Spectral-gauge cutting-plane algorithm}
\label{sec:algorithm}
In Algorithm~\ref{alg:master_oracle_gauge}, we integrate the separation oracle~\eqref{eq:gphi} into a cutting-plane method for solving~\eqref{eq:sdp}. Starting from an outer approximation $\mathcal{R}\subseteq\mathcal{A}\cap\Plin$ defined on 
a cone $\mathcal A\supseteq\mathbb S^n_+$, the proposed algorithm solves a sequence of master problems $(\mathcal{M}_i)_i$
\begin{equation}\label{eq:Mi}
\begin{aligned}
\mathcal{M}_i:\ \min_{\mathbf{X}\in\Sn}\quad & \ip{\mathbf{C}}{\mathbf{X}}\\
\text{s.t.}\quad
& \mathbf{X}\in\mathcal{R}\subseteq\mathcal{A}\cap\Plin,\\
& \ip{\mathbf{S}}{\mathbf{X}} \ge 0 \quad \forall\,\mathbf{S}\in\mathcal{S}_i.
\end{aligned}
\end{equation}
and adds a spectral-gauge cut (for a given symmetric gauge $\phi$) whenever the current solution is not PSD (within a tolerance $\varepsilon$). 
In what follows, we discuss possible initial relaxations $\mathcal{A}$ and  
boundedness conditions. We prove the convergence of the algorithm in the general context, and 
illustrate specific cases.



\medskip

\begin{algorithm}[H]
\caption{Spectral-gauge cutting-plane method}
\label{alg:master_oracle_gauge}

\KwRequire{
$\mathbf{C}\in\Sn$; polyhedron $\Plin\subseteq\Sn$; outer-approximation $\mathcal{A}$ with $\mathbb{S}^n_+\subseteq \mathcal{A}\subseteq \Sn$; symmetric gauge $\phi$; tolerance $\varepsilon\ge 0$.
}
\KwEnsure{
An $\varepsilon$-solution $X$ to~\eqref{eq:sdp}, i.e., with $\blambda_{\min}(X)>-\varepsilon$.

}

$\mathcal{S}_0\gets\emptyset$\;

\For{$i\gets 0,1,2,\dots$}{
  Solve $\mathcal{M}_i$: get optimizer $\mathbf{X}^{(i)}$ with eigendecomposition\;
  \[
    \mathbf{X}^{(i)}=\mathbf{Q}^{(i)}\operatorname{diag}\!\bigl(\lambda^{(i)}_1,\dots,\lambda^{(i)}_n\bigr)(\mathbf{Q}^{(i)})^\top,
    \quad \lambda^{(i)}_1\le\cdots\le\lambda^{(i)}_n.
  \]

  Set $\blambda_{-}^{(i)}:=
  (-\lambda^{(i)})_+$
  and compute 
  $$\boldsymbol{z}^{(i)}\in\argmax\left\{(\blambda_{-}^{(i)})^\top\boldsymbol{z}:\ \boldsymbol{z}\ge \mathbf{0},\ \phi(\boldsymbol{z})\le 1, \ \operatorname{supp}(z)\subseteq \operatorname{supp}(\lambda_-^{(i)})\right\}.$$

  Form $\mathbf{S}^{(i)}:=\mathbf{Q}^{(i)}\operatorname{diag}\!\bigl({\boldsymbol{z}}^{(i)}\bigr)(\mathbf{Q}^{(i)})^\top$.

  \eIf{$\blambda_1^{(i)}>-\varepsilon$}{
    \KwRet{$\mathbf{X}^{(i)}$}\;
  }{
    $\mathcal{S}_{i+1}\gets \mathcal{S}_i \cup\{\mathbf{S}^{(i)}\}$\;
  }
}
\end{algorithm}

\subsection{Outer-approximations of the PSD cone}\label{subsec:outer-approximations}
To establish the convergence of the algorithm, we assume that there exists $T>0$ such that every feasible solution $\mathbf X$ of
\eqref{eq:sdp} satisfies $\operatorname{tr}(\mathbf X)\le T$. 
Such a bound is explicit in the computation, if not in the formulation.
We then enforce this condition in the master problems $\mathcal{M}_i$
\[
\mathcal{R}:=
\Bigl\{
\mathbf{X}\in\Sn:
\mathbf{X}\in\Plin,\
\mathbf{X}\in\mathcal{A},\ \operatorname{tr}(\mathbf{X})\le T
\Bigr\},
\]
with convex cone $\mathcal{A}\supseteq\mathbb{S}^n_+$.
We also assume 
that every nonzero element $X$ of $\mathcal{A}$ has a strictly positive trace, so that it can be normalized $\|\mathbf{X}\|_{\mathcal X}\le \alpha_{\mathcal X}(\mathcal A)\,\operatorname{tr}(\mathbf X)$, where
\[
\alpha_{\mathcal X}(\mathcal A)
:=
\sup\Bigl\{
\|\mathbf{X}\|_{\mathcal X}:\ \mathbf{X}\in\mathcal A,\ \operatorname{tr}(\mathbf{X})=1
\Bigr\}
\]
measures the size of cone $\mathcal A$ in an arbitrary norm $\|\cdot\|_{\mathcal X}$ on $\mathbb S^n$. 

Let $\mathcal F=\Plin\cap\Snp \neq \varnothing$ denote the feasible set of the original SDP in~\eqref{eq:sdp}.
If $\alpha_{\mathcal X}(\mathcal A)<\infty$, then $\mathcal R$ is a compact set containing
$\mathcal F$ and the feasible sets
\begin{equation}\label{eq:compact_nesting}
\mathcal{R}^{(i)}:=\Bigl\{\mathbf{X}\in\mathcal{R}:\ \ip{\mathbf{S}^{(t)}}{\mathbf{X}}\ge 0,\ t=1,\dots,i-1\Bigr\}, \quad \mathcal R^{(1)} \supseteq \mathcal R^{(2)} \supseteq \cdots \supseteq \mathcal F.
\end{equation}
Hence, every master problem $\mathcal{M}_i$ is feasible and attains an optimum.

The dual factor-width-$m$ cone hierarchy in~\eqref{eq:m-local}
provides us with a set of candidates for $\mathcal{A}$ meeting our assumptions:
$\mathcal{LP}\supseteq \mathcal{SOC}=\mathbb{S}^{n,2}_+ \supseteq \mathbb{S}^{n,3}_+ \supseteq \cdots \supseteq \mathbb{S}^{n,n}_+
=\Snp$,
where the case $m=2$ is the second-order cone
\begin{equation}\label{eq:FW2_SOC}
\mathcal{SOC}
:=
\Bigl\{
\mathbf{X}\in\mathbb{S}^n:\ X_{jj}\ge 0\ \ \forall j,\ \ X_{jl}^2\le X_{jj}X_{ll}\ \ \forall j\neq l
\Bigr\},
\end{equation}
and $\mathcal{LP}$ is the linear relaxation of \eqref{eq:FW2_SOC} obtained from the inequality of arithmetic and geometric means, namely
\begin{equation}\label{eq:LP_def}
\mathcal{LP}
:=
\Bigl\{
\mathbf{X}\in\mathbb{S}^n:\ X_{jj}\ge 0\ \ \forall j,\ \
X_{jj}+X_{ll}\pm 2X_{jl}\ge 0\ \ \forall j\neq l
\Bigr\}.
\end{equation}
Note that $\alpha_{\mathcal X}(\mathcal{LP})<\infty$, hence
\begin{equation}\label{eq:hierarchy-outer-approx}
\alpha_{\mathcal X}(\mathbb S^n_+)
\le
\alpha_{\mathcal X}(\mathbb{S}^{n,n-1}_+)
\le \cdots \le
\alpha_{\mathcal X}(\mathbb{S}^{n,3}_+)
\le
\alpha_{\mathcal X}(\mathcal{SOC})
\le
\alpha_{\mathcal X}(\mathcal{LP})<\infty.
\end{equation}

\subsection{Algorithm convergence analysis}\label{subsec:accumulation}

Recall the compact set $\uball_\phi$ and $g_\phi$ from Definition~\ref{def:Bgauge}. The convergence analysis depends on two geometric constants: $\alpha_{\mathcal X}(\mathcal A)$ of the chosen $\mathcal A$, and the uniform bound
\[
L_{\mathcal X}(\phi):=\max_{\mathbf S\in\sball_\phi}\|\mathbf S\|^\circ_{\mathcal X}.
\]

$L_{\mathcal X}(\phi)$ is a finite constant and Cauchy-Schwarz inequality~\eqref{eq:cs} yields
\begin{equation}\label{eq:Lipschitz_gauge}
|\ip{\mathbf{S}}{\mathbf{X}_1}-\ip{\mathbf{S}}{\mathbf{X}_2}|
\le
L_{\mathcal X}(\phi)\,\|\mathbf{X}_1-\mathbf{X}_2\|_{\mathcal X}
\quad
\forall\,\mathbf X_1,\mathbf X_2\in\Sn,\ \forall\,\mathbf S\in\sball_\phi.
\end{equation}
In particular,
$L_{\mathcal X}(\phi)$ is monotone under enlargements of $\uball_\phi$, so it is nondecreasing
along the $\ell_p$ gauge and rank-$k$ hierarchies.

\begin{lemma}[Finite $\varepsilon$-termination bound]\label{lem:eventual_eps_psd_gauge}
Assume that $\alpha_{\mathcal X}(\mathcal A)<\infty$. Let
\[
d:=\dim(\mathbb S^n)=\frac{n(n+1)}{2},
\quad
N^\varepsilon_\mathcal{X}(\mathcal{A},\phi)
:=
1+\left\lceil
\left(
1+\frac{2\,\alpha_{\mathcal X}(\mathcal A)\,T\,L_{\mathcal X}(\phi)\,\phi(e_1)}{\varepsilon}
\right)^d
\right\rceil,
\]
where $e_1$ is any canonical basis vector. Then there exists $i\in [N^\varepsilon_\mathcal{X}(\mathcal{A},\phi)]$ such that
one has $\lambda_{\min}(\mathbf X^{(i)})>-\varepsilon$. Equivalently,
Algorithm~\ref{alg:master_oracle_gauge} terminates, with stopping tolerance
$\varepsilon > 0$, after at most $N^\varepsilon_\mathcal{X}(\mathcal{A},\phi)$ iterations.
\end{lemma}

\begin{proof}
Fix $\varepsilon>0$ and set $\eta:=\varepsilon/\phi(e_1)$. Since $e_i/\phi(e_1)$ is
feasible in the dual-norm formula~\eqref{eq:dual-nonneg} for every $i$, we have
$\phi^\circ(y)\ge \|y\|_\infty/\phi(e_1)$ for every $y\in\mathbb R^n_+$. Hence, by
Theorem~\ref{thm:spectral-gauge},
$g_\phi(\mathbf X)\le -\|\lambda_-(\mathbf X)\|_\infty/\phi(e_1)$ for every
$\mathbf X\in\Sn$. If Algorithm~\ref{alg:master_oracle_gauge} has not terminated at
iteration $i$, then $\lambda_{\min}(\mathbf X^{(i)})\le -\varepsilon$, and therefore
$g_\phi(\mathbf X^{(i)})\le -\eta$.

Choose
$\mathbf S^{(i)}\in\argmin_{\mathbf S\in\sball_\phi}\ip{\mathbf S}{\mathbf X^{(i)}}$.
Then \eqref{eq:Lipschitz_gauge} gives, for every $\mathbf Y\in\Sn$,
\[
\ip{\mathbf S^{(i)}}{\mathbf Y}
\le
g_\phi(\mathbf X^{(i)})+L_{\mathcal X}(\phi)\,\|\mathbf Y-\mathbf X^{(i)}\|_{\mathcal X}.
\]
Define $R=\eta/2L_{\mathcal X}(\phi)$.
Hence, every $\mathbf Y$ satisfying
$\|\mathbf Y-\mathbf X^{(i)}\|_{\mathcal X}< 2R$ violates the
cut added at iteration $i$, so no later iterate can belong to that open ball.
Therefore the family of balls
$\uball_{\mathcal X}\!\left(\mathbf X^{(i)},\, R\right)$
centered at iterates $X^{(i)}$ failing the stopping test, i.e., for all $i\in\mathcal N_\varepsilon
:=
\{i\in\mathbb N:\lambda_{\min}(\mathbf X^{(i)})\le -\varepsilon\}$, have pairwise disjoint interiors, and they all share the same volume $\operatorname{vol}_d(\mathcal B_{\mathcal X}(\mathbf 0,R))= c_{\mathcal X}R^d$, in Lebesgue measure, for a constant $c_{\mathcal X}>0$ related to $d$-space $(\Sn,\mathcal{X})$.
By~\eqref{eq:compact_nesting}, all iterates lie in
$\mathcal R\subseteq \uball_{\mathcal X}\!\left(\mathbf 0,\alpha_{\mathcal X}(\mathcal A)\,T\right)$
so the family of disjoint balls is included in
$
\uball_{\mathcal X}\!\left(\mathbf 0,\alpha_{\mathcal X}(\mathcal A)\,T+
R\right)$.
Then, comparing volumes gives, for every finite subset $J\subseteq\mathcal N_\varepsilon$,
$$|J|\,c_{\mathcal X}R^d
\le
c_{\mathcal X}(\alpha_{\mathcal X}(\mathcal A)T+R)^d.$$
Since $c_{\mathcal X}>0$, dividing by $c_{\mathcal X}R^d$ gives
\[
|J|
\le
\left(
1+
\frac{\alpha_{\mathcal X}(\mathcal A)T}{R}
\right)^d
=
\left(
1+\frac{2\,\alpha_{\mathcal X}(\mathcal A)\,T\,L_{\mathcal X}(\phi)\,\phi(e_1)}{\varepsilon}
\right)^d.
\]
Taking the supremum over all finite subsets $J\subseteq\mathcal N_\varepsilon$, we obtain
\begin{align*}
|\mathcal N_\varepsilon|  =\sup_{J\subseteq\mathcal N_\varepsilon}\bigl\{|J|: \ J \text{ finite}\bigr\} < 1+\left\lceil
\left(
1+\frac{2\,\alpha_{\mathcal X}(\mathcal A)\,T\,L_{\mathcal X}(\phi)\,\phi(e_1)}{\varepsilon}
\right)^d
\right\rceil = N^\varepsilon_\mathcal{X}(\mathcal{A},\phi).
\end{align*}
Thus, not all of the first $N^\varepsilon_\mathcal{X}(\mathcal{A},\phi)$ iterates can fail the stopping test in Algorithm~\ref{alg:master_oracle_gauge}.
\end{proof}

\begin{theorem}[Convergence of cut method]\label{thm:convergence_gauge}
Let $\mathcal{A}\subseteq\Sn$ be a closed convex cone with
$\Snp\subseteq\mathcal{A}$, $\alpha_{\mathcal X}(\mathcal{A})<\infty$, and $tr(X)>0$ for all $X\in\mathcal{A}\setminus\{0\}$. Then, for any symmetric gauge $\phi$, Algorithm~\ref{alg:master_oracle_gauge}, run
with stopping tolerance $\varepsilon=0$, either terminates finitely at an optimal solution of the
SDP~\eqref{eq:sdp}, or every limit point of the generated infinite sequence
$(\mathbf{X}^{(i)})_{i}$ is an optimal solution of the SDP~\eqref{eq:sdp}.
\end{theorem}

\begin{proof}
If Algorithm~\ref{alg:master_oracle_gauge} terminates finitely, the returned
point satisfies $\lambda_{\min}(\mathbf X^{(i)})>0$. Since $X^{(i)}$ solves the relaxed problem $\mathcal{M}_i$ in \eqref{eq:Mi} and $\lambda_{\min}(\mathbf X^{(i)})>0$, $X^{(i)}$ is feasible for~\eqref{eq:sdp}. Hence, $X^{(i)}$ solves~\eqref{eq:sdp}.
Assume now that the algorithm generates an infinite sequence. Let $\bar{\mathbf X}$ be
any limit point of $\{\mathbf X^{(i)}\}_{i\ge 1}$, and take a subsequence
$\mathbf X^{(i_j)}\to \bar{\mathbf X}$. Define
$v_i:=\ip{\mathbf C}{\mathbf X^{(i)}}$. Since
$\mathcal R^{(i+1)}\subseteq \mathcal R^{(i)}$, the sequence $(v_i)_i$ is
nondecreasing; and since
$\mathcal F\subseteq \mathcal R^{(i)}$ for all $i$, it is bounded above by
\[
v^\star:=\min\{\ip{\mathbf C}{\mathbf X}:\ \mathbf X\in\mathcal F \neq \varnothing\}.
\]
Hence $v_i\to \bar v\le v^\star$ for some $\bar v$.

Next, the proof of Lemma~\ref{lem:eventual_eps_psd_gauge} shows that, for each
$p\in\mathbb N$, only finitely many iterates satisfy $\lambda_{\min}(\mathbf X^{(i)})\le -1/p$. Since the zero-tolerance run does not terminate, we have
$\lambda_{\min}(\mathbf X^{(i)})\le 0$ for every $i$, and therefore, along the
subsequence chosen above, $\lambda_{\min}(\mathbf X^{(i_j)})\to 0$. By continuity of $\lambda_{\min}$, $\lambda_{\min}(\bar{\mathbf X})=0$, hence $\bar{\mathbf X}\succeq \mathbf 0$. Since all iterates lie in the closed set $\mathcal R$ and the affine constraints are preserved in the limit, we have $\bar{\mathbf X}\in\mathcal F$. Finally, continuity of the objective gives
\[
\ip{\mathbf C}{\bar{\mathbf X}}
=
\lim_{j\to\infty}\ip{\mathbf C}{\mathbf X^{(i_j)}}
=
\lim_{j\to\infty} v_{i_j}
=
\bar v
\le v^\star.
\]
But $\bar{\mathbf X}\in\mathcal F$ implies $\ip{\mathbf C}{\bar{\mathbf X}}\ge v^\star$.
Therefore $\ip{\mathbf C}{\bar{\mathbf X}}=v^\star$, and $\bar{\mathbf X}$ is optimal.
\end{proof}


\subsection{Illustration and comparison}\label{subsec:hierarchies}The above analysis follows the proof in~\cite{bertsimas_cory_wright_2020}, extending it to any arbitrary norm $\mathcal{X}$ in place of the Frobenius norm. This allows us, now, to exhibit a strict hierarchy between the considered instantiations of the framework, $\ell_p$-gauge cuts for $p\in[1,\infty]$,  rank-$k$ cuts for $k\in[n]$, and factor-width-$m$ relaxations for $m\in[n]$, whereas the termination bounds $N^\varepsilon_\mathcal{X}$ may collapse to equalities in the Frobenius setting.
Let us define
$$\|X\|_{\mathcal X}:=\|X\|_{\mathcal X_0}+\sum_{m=3}^n \delta_m \beta_m(X),
\text{ where } \|X\|_{\mathcal X_0}
:=
\max\bigl\{\|Xu\|_2 : u\in\mathbb R^n,\ \|u\|_2=1\bigr\}$$ 
is the operator norm of $X\in\mathbb S^n$ induced by $\ell_2$ on $\mathbb R^n$~\cite[Def.~5.6.1]{horn_johnson_2013}, and 
\[
\beta_m(X):=\max_{\substack{I\subseteq[n]\\ |I|=m}}\|X_I\|_{\mathcal{X}_0}^\circ,
\]
with $m=3,\dots,n$, and $0<\delta_n<\cdots<\delta_3<1$.



Let $1\le p\le\infty$ and $q$ be the dual exponent. Since $\|S\|_{\mathcal X_0} ^\circ=\operatorname{tr}(S)$ for $S\in\Snp$,
\[
L_{\mathcal X_0}(\ell_p)
=
\max\Bigl\{\mathbf 1^\top z : z\in\mathbb R_+^n,\ \|z\|_p\le1\Bigr\}
=
n^{1/q}
\]
Moreover, by Theorem~\ref{cor:sparsified_oracle}, one has, for any rank $k\in[n]$
\[
L_{\mathcal X_0}(\ell_p^{\langle k\rangle})
=
\max\Bigl\{\mathbf 1^\top z : z\in\mathbb R_+^n,\ \|z\|_p\le1,\ \|z\|_0\le k\Bigr\}
=
k^{1/q}.
\]
Therefore, 
 both sequences $(N^\varepsilon_\mathcal{X}(\mathcal{A},\ell_p))_{p\ge 1}$ and $(N^\varepsilon_\mathcal{X}(\mathcal{A},\ell^{\langle k\rangle}_p))_{k\ge 1}$ for a fixed $p>1$, are strictly increasing.
Note in passing that, for $p=1$, one has $(\ell_1^{\langle k\rangle})^\circ(y)=\|\operatorname{trun}_k(y)\|_\infty=\|y\|_\infty$ so the rank limit has no impact on spectral-$\ell_1$ separation.

We now build the coefficients of $\mathcal{X}$ so that,  for any fixed gauge $\phi$, the sequence $(N^\varepsilon_\mathcal{X}(\mathcal{A}_m,\phi))_{m\ge 1}$ is strictly decreasing along the factor-width relaxation hierarchy \eqref{eq:hierarchy-outer-approx}: $\mathcal{A}_1=\mathcal{LP}$, $\mathcal{A}_m=\mathbb{S}^{n,m}_+$ for $m\in\{2,\ldots,n\}$.
First, 
consider
\[
W^{(2)}:=
\begin{pmatrix}
a & c\\
c & 1-a
\end{pmatrix},
\quad
0<a<1,
\quad
\sqrt{a(1-a)}<c\le \frac12,
\]
then $W^{(2)}\in \mathcal{LP}\setminus \mathbb{S}^{n,2}_+$, $\operatorname{tr}(W^{(2)})=1$, 
$\|W^{(2)}\|_{\mathcal X_0}>1$, and thus $\alpha_{\mathcal X_0}(\mathcal{LP})>1$.
On the other hand, if $X\in \mathbb{S}^{n,2}_+=\mathcal{SOC}$ and $\operatorname{tr}(X)=1$, then for every
$u\in\mathbb R^n$ with $\|u\|_2=1$,
$$|u^\top Xu| 
\le
\sum_{j,l} |X_{jl}|\,|u_j|\,|u_l|
\le
\Bigl(\sum_j \sqrt{X_{jj}}\,|u_j|\Bigr)^2
\le \Bigl(\sum_j X_{jj}\Bigr)\Bigl(\sum_j u_j^2\Bigr) = 1.$$
Hence $\alpha_{\mathcal X_0}(\mathbb{S}^{n,m}_+)=1$ for every $m=2,\dots,n$, since
$\alpha_{\mathcal X_0}(\mathbb{S}^{n,m}_+) \le \alpha_{\mathcal X_0}(\mathbb{S}^{n,2}_+)\le 1$
and equality $u^\top Xu=1$ is attained by any rank-one matrix $X$ with trace $1$.

Given $m\in\{2,\dots,n-1\}$, if $X\in \mathbb{S}^{n,m+1}_+$ and $\operatorname{tr}(X)=1$, then every
$(m+1)\times(m+1)$ principal submatrix $I$ is PSD, so $\|X_I\|_{\mathcal X_0} ^\circ=\operatorname{tr}(X_I)\le 1$, then $\beta_{m+1}(X)\le 1$.
Conversely, we build a matrix $W\in\mathbb{S}^{n,m}_+\setminus\mathbb{S}^{n,m+1}_+$  such that $\beta_{m+1}(W)> 1$, as follows:
\[
\widehat{W}
:=
\frac1{m+1}\bigl((1-\rho_m)I_{m+1}+\rho_m J_{m+1}\bigr),
\quad
-\frac1{m-1}\le \rho_m<-\frac1m,
\]
where $J_{m+1}$ is the all-ones matrix, and embed $\widehat{W}$ as a principal block of
an $n\times n$ matrix $W$. Then $\operatorname{tr}(W)=1$ and every $m\times m$
principal submatrix of $W$ is PSD, but $W\notin \mathbb{S}^{n,m+1}_+$ because
$\widehat W$ has eigenvalue $(1+m\rho_m)/(m+1)<0$. Moreover,
\[
\beta_{m+1}(W)
=
\|\widehat W\|^\circ_{\mathcal{X}_0}
=
\frac{m-1-2m\rho_m}{m+1}
>
1.
\]
Therefore  the coefficients $\delta_m$ can be chosen recursively for $m\ge 3$ so that the sequences $(\alpha_{\mathcal X}(\mathcal{A}_m))_{m\ge 1}$ and thus $(N^\varepsilon_\mathcal{X}(\mathcal{A}_m,\phi))_{m\ge 1}$ are strictly decreasing.

\section{Numerical experiments}\label{sec:experiments}

In this section, we analyze our proposed cutting-plane framework empirically across three aspects: versatility (the ease of implementing various spectral-gauge templates), relative performance (the individual stren\-gths of different templates, including standard eigencuts), and absolute performance (comparison with an off-the-shelf interior-point SDP solver). We aim to be neither exhaustive, given the huge number of combinations of the SDP problem with outer-approximation cones and gauge templates, nor exclusive, that is, selecting the best combination and refining the method specifically for it. 

Hence, our implementation of Algorithm~\ref{alg:master_oracle_gauge} follows the basic Kelley scheme, without any cut-management or acceleration strategy, adding only one cut at each iteration, namely the one with the largest violation at the current iterate.
We compare various instantiations of this framework, considering either the LP or SOC relaxations combined with some of the studied spectral-gauge templates, namely eigenvalue, Frobenius, and nuclear cuts, without or with rank restrictions. We evaluate the dual bounds computed by these combinations on two distinct classes of NP-hard problems commonly addressed through SDP relaxations: box-constrained quadratic programs (BoxQP)~\cite{burer_vandenbussche_2009,dey_etal_2022}, and sparse principal component analysis (SPCA)~\cite{bertsimas_cory_wright_2020}. 

All numerical experiments are conducted on an Apple M3 Pro with a 5-core 4.05\,GHz performance CPU, a 6-core 2.75\,GHz efficiency CPU, 
and 36\,GB RAM. The code\footnote{The code is available at \url{https://github.com/sofdem/sdpgauge26}.} is implemented in Python~3.12, and the eigenvalue decompositions called for separation  are computed with \texttt{numpy.linalg.eigh}.
The master problem $\mathcal{M}_i$~\eqref{eq:Mi} is solved using the homogeneous barrier algorithm in Gurobi~v13.0, on both LP and SOC formulations. 
We measure the performance 
of Algorithm~\ref{alg:master_oracle_gauge}
using the \textit{gap closed}~\cite{dey_etal_2022}:
\begin{align*}
    GC^T = 100 \cdot \frac{M\!P_{T}-LP}{best-LP},
\end{align*}
where $M\!P_{T}$ is the value of the last master problem solved by Algorithm~\ref{alg:master_oracle_gauge} before the time limit $T$ and $LP$ is the optimal value of the initial LP relaxation~\eqref{eq:LP_def}. For small instances (typically, $n\leq 100$), $best$ is the optimal value of the SDP relaxation~\eqref{eq:sdp}, computed by running the primal-dual interior-point solver \texttt{MOSEK}~v11.1.10 with tolerance $10^{-6}$. 
We indicate with $t_{sdp}$ the corresponding \texttt{MOSEK} computation time, and evaluate the gap closed at $T\in\{t_{sdp},60s\}$.
Since interior-point methods become expensive in time and memory as $n$ grows, for larger instances, 
we define $best$ as the largest dual bound obtained by any tested cut strategy within 30 minutes. 
We then evaluate the gap closed relative to this virtual best strategy at $T\in\{60s,1800s\}$.

\paragraph{Cut strategies} 
We implement several spectral-gauge cut templates: eigenvalue cuts (Corollary~\ref{cor:rank1}), Frobenius cuts (Corollary~\ref{cor:frob}), nuclear cuts (Corollary~\ref{cor:projector}).
In the tables below, the results obtained using these strategies are reported in the columns labeled $\ell_1$, $\ell_2$, and $\ell_\infty$, respectively.
We also generate spectral cuts
 using a dynamic rank parameterization (Theorem~\ref{cor:sparsified_oracle}). This is implemented by selecting the negative eigenvalues belonging to the interval $[\lambda_1(X),0.9\lambda_1(X)]$, that is, those closest to the smallest one. The columns labeled as $\ell_*^{90\%}$ refer to the results obtained with this strategy.  
In this same dynamic rank setting, we additionally consider selecting a varying number of the most negative eigenvalues. This number changes across iterations and is selected from a discrete set whose dimension depends on the problem dimension. For instance, for $n=20$, we randomly select a number in $[1,5,10]$ or all eigenvalues, whereas for $n=250$, we consider $[1,5,10,20]$ or all eigenvalues. 
We use the notation $\ell_*^{\textrm{rand}}$ for this last case.
\subsection{Box-constrained quadratic programs}
We consider the problem
$$\min x^\top Q x + c^\top x: x\in [0,1]^n,$$
where $Q \in \Sn$ is not positive-definite and $c \in \R^n$. A standard way to obtain a tractable bound is to use an SDP relaxation~\cite{dey_etal_2022,locatelli_etal_2025}. The lifting approach linearizes the quadratic form $x^\top Q x = \tr(Q xx^\top) = \ip{Q}{xx^\top}$ by introducing a symmetric matrix variable $X=xx^\top$. Since the set of matrices of the form $xx^\top$ is nonconvex, the equality $X=xx^\top$ is relaxed to $X-xx^\top\succeq 0$. The resulting SDP relaxation is
$$\min \ip{Q}{X}+c^\top x\ :\ 
X -xx^\top \succeq 0,\  x\in[0,1]^n, \ X\in\Sn.$$
We relax the PSD condition as either \eqref{eq:FW2_SOC} or \eqref{eq:LP_def}, and add McCormick inequalities:
\[
X_{jl} \geq 0,\quad
X_{jl} \leq x_j,\quad
X_{jl} \leq x_l,\quad
X_{jl} \geq x_j+x_l-1.
\]
The experiments rely on the BoxQP library: the ``basic'' instances with $n\in[20,60]$ from~\cite{vandenbussche_nemhauser_2005} and newly benchmarked sets from the generator in~\cite{burer_vandenbussche_2009}, made of 3 instances for each dimension $n\in\{30,250\}$ and  density $d\in\{25\%, 50\%, 75\%, 100\%
\}$.\footnote{The instances and generator are available at \url{https://github.com/sburer/BoxQP_instances/}.}

\paragraph{Comparison with an SDP interior-point solver}
We first evaluate the dual bounds obtained by Algorithm~\ref{alg:master_oracle_gauge} using the different strategies above for the same time $T=t_{sdp}$ required by \texttt{MOSEK} to compute the SDP bound. In Table~\ref{tab:boxqp-basic-gc}, we report the gap closed value on the ``basic" instances~\cite{vandenbussche_nemhauser_2005}, averaged over the number of instances (column $\#$) per dimension $n$, and in Table~\ref{tab:boxqp-n30-gc} (top part), the gap closed on the generated set with dimension $n=30$, averaged over the 3 instances per density $d$.




For the smallest instances ($n=20$), the LP-based cutting-plane methods close a large fraction of the SDP gap within the \texttt{MOSEK} running time.
For all other instances, except the largest one, the Frobenius cuts are particularly competitive in terms of the closed gap. In the setting of a very short  time limit, the use of the SOC relaxation is discouraged. In fact, only a few iterations of Algorithm~\ref{alg:master_oracle_gauge} are performed, leading to smaller closed gap values. As an example, the number of iterations for the case $n=20$ is an average of $6$ for SOC with eigenvalue cuts, compared with $106$ for the LP relaxation. The notation `$-$' in the table indicates that the allowed computation time was not sufficient to perform even a single iteration for at least one instance.
\begin{table}[]
\footnotesize
  \caption{BoxQP ``basic'' instances~\cite{vandenbussche_nemhauser_2005}: average gap closed $GC^T$ (\% of the SDP value) for $T=t_{sdp}$.}
    \label{tab:boxqp-basic-gc}
    \centering
    \begin{tabular}{rrrr|rrrr|rr}
    \multicolumn{4}{l}{} &\multicolumn{4}{c}{LP} &\multicolumn{2}{c}{SOC}\\
       $n$ &$d$ &$\#$ &$T$
       &$\ell_1$ &$\ell_\infty$ &$\ell_2$ &$\ell^{90\%}_\infty$
       &$\ell_1$ &$\ell_2$\\
        \hline
       20 &100  &3 &0.3  &{\bf 99.9} &99.8 &{\bf 99.9} &{\bf 99.9} &58.9 &91.7\\
       30 &[60,100]  &15 &2.3  &94.8 &94.3 &{\bf 95.6} &95.0 &46.4 &87.2\\
       40 &[30,100]  &24 &9.0  &88.6 &86.2 &{\bf 90.8} &89.5 &30.5 &83.2\\
       50 &[30,50]  &9 &25.9 &71.8 &67.0 &{\bf 75.9} &74.4 &- &-\\
       60 &20  &3 &56.8 &33.7 &28.0 &38.6 &{\bf 42.6} &- &-\\
       \hline
    \end{tabular}
\end{table}

\begin{table}[]
\footnotesize
    \caption{BoxQP instances with $n=30$: 
    average gap closed $GC^T$ (\% of the SDP value) for $T=t_{sdp}$ (top) and $T=60$ seconds (bottom).}
    \label{tab:boxqp-n30-gc}
    \centering
    \begin{tabular}{rrr|rrrrrr|rrr}
    \multicolumn{3}{l}{} &\multicolumn{6}{c}{LP} &\multicolumn{3}{c}{SOC}\\
       $d$ &$\#$ &$T$ 
       &$\ell_1$ &$\ell_\infty$ &$\ell_2$ &$\ell^{90\%}_\infty$ &$\ell^{90\%}_2$ &$\ell_\infty^{rand}$
       &$\ell_1$ &$\ell_\infty$ &$\ell_2$\\
        \hline
       25  &3 &1.4  &86.9 &85.7 &{\bf 90.9} &89.2 &87.7 &88.4 &25.7 &53.2 &61.0\\
       50  &3 &1.8  &90.4 &89.2 &91.4 &90.7 &90.9 &{\bf 91.6} &4.1 &71.1 &79.5\\
       75  &3 &3.0  &94.0 &94.1 &95.6 &94.2 &94.3 &{\bf 95.7} &35.3 &81.4 &86.7\\
       100 &3 &3.0 &95.7 &95.4 &{\bf 96.9} &95.9 &96.1 &96.8 &51.2 &84.4 &87.2\\
        \hline
25 &3 &60 &{\bf 100} &{\bf 100} &{\bf 100} &{\bf 100} &{\bf 100} &{\bf 100} &95.2 &94.6 &96.1\\
50 &3 &60 &{\bf 99.2} &96.2 &97.2 & 98.9 & 98.9 &98.2 &92.8 &92.4 &94.3\\
75 &3 &60 &{\bf 99.7} &98.7 &99.0 & 99.6 & 99.6 &99.3 &96.1 &95.7 &96.6\\
100 &3 &60 &{\bf 99.6} &98.5 &99.0 & 99.5 & 99.5 &99.3 &96.8 &96.3 &97.4\\
\hline
    \end{tabular}
\end{table}


On the generated set $n=30$, we can observe the relationship between algorithm performance and instance density. 
Regardless of the cut strategy and relaxation type, the denser the instance, the larger the gap closed. Indeed, as observed in \cite{dey_etal_2022}, when the objective function has many zero coefficients, the objective value tends to vary less across iterations of a cutting-plane algorithm. The Frobenius cuts outperform the other strategies in both the LP and SOC relaxations, and the randomized rank-setting with LP relaxation appears to be equally robust.


\paragraph{Convergence on the small instances $(n=30)$}
We now analyze how the dual bound evolves over the iterations of the cutting-plane algorithm.
At the bottom part of Table~\ref{tab:boxqp-n30-gc}, we report the average gap closed when the allowed computation time is set to $T=60$ seconds on the generated set $n=30$.
In the LP setting, the eigencuts strategy proves to be the most effective across the different density values, and all cut families could close the gap in less than one minute for the sparser instances, except for the full-rank cuts on instance \texttt{030-025-2} ($\ell_2$ stops after 120 seconds and $\ell_\infty$ after 263 seconds).
This instance is particularly challenging for the cutting-plane algorithm as the gap closed at $T = t_{sdp}$ is around $72\%$ under the best cut strategy, whereas it is around $99.8\%$ for the other two instances with $d=25\%$.

\begin{figure}
    \centering
    \subfigure[]{
    \includegraphics[width=0.45\linewidth]{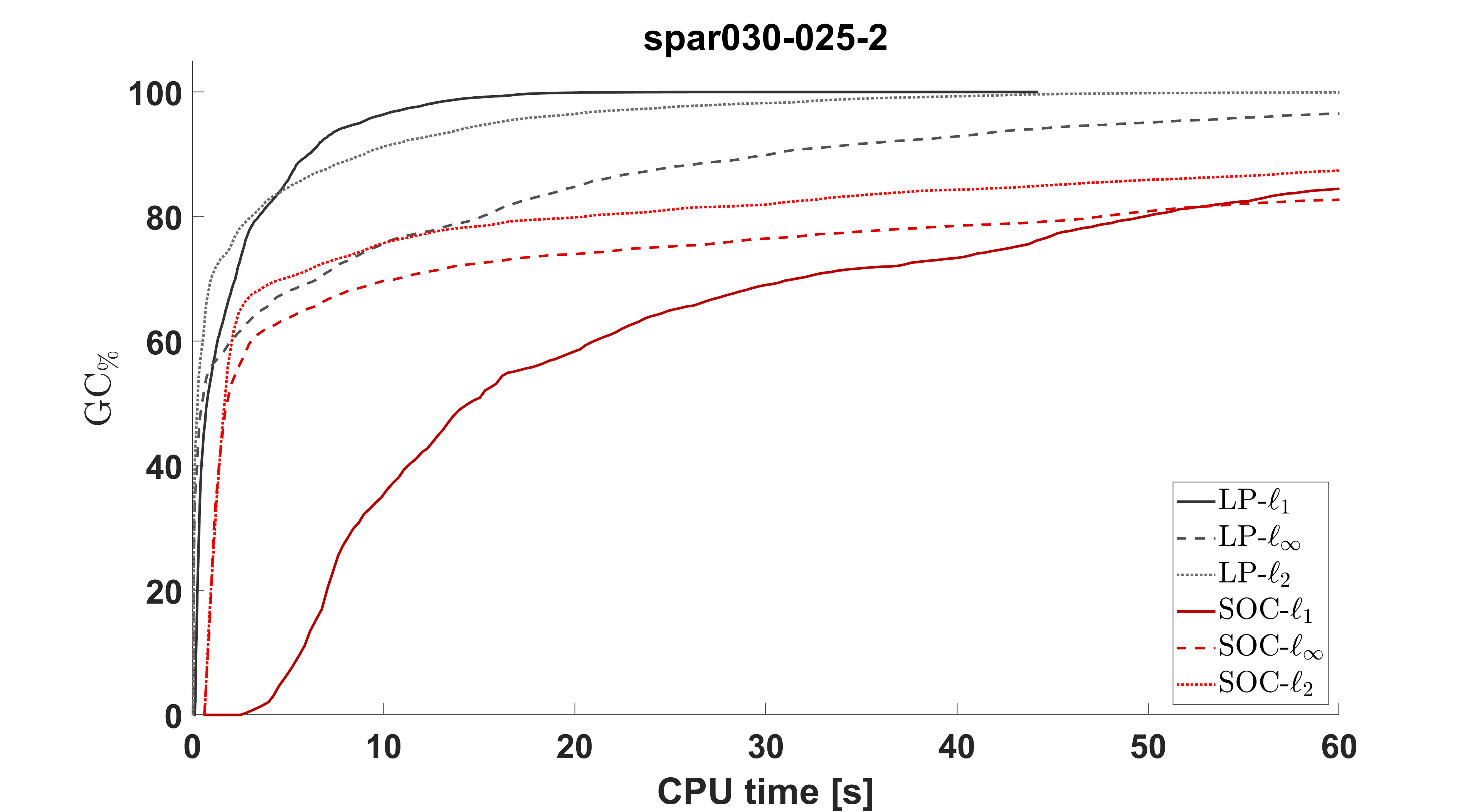}}
    \subfigure[]{
    \includegraphics[width=0.45\linewidth]{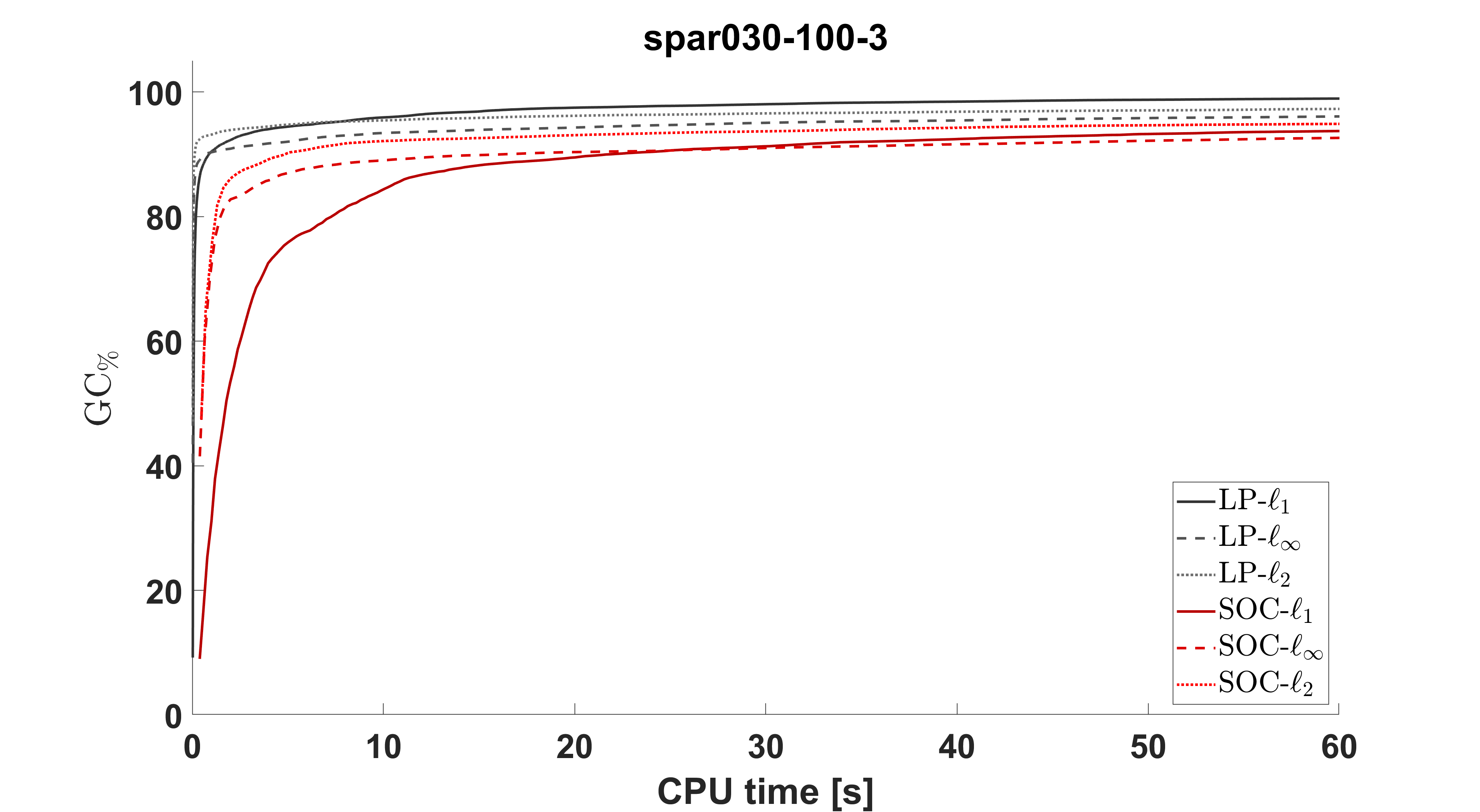}}
    \caption{Gap closed evolution for two BoxQP instances with $n=30$ considering $T = 60$ seconds.}
    \label{fig:boxqp-n30-gc-cpu_sel}
\end{figure}

Figure~\ref{fig:boxqp-n30-gc-cpu_sel} illustrates the evolution of the gap closed over the computation time, for this instance, on the left, and for the full dense case \texttt{030-100-3}, on the right. 
We observe that the LP relaxation dominates the SOC relaxation. Regarding cuts, Frobenius cuts are the most effective across the first seconds. 
After at most $10$ seconds, the LP relaxation with eigencuts allows the largest gap closed, and it is the first configuration to terminate, in 44 seconds, on the sparse instance.

We also evaluate the performance of 
the different cut templates in the LP setting 
at convergence, by running Algorithm~\ref{alg:master_oracle_gauge} on set $n=30$ for at most two hours.
In Table~\ref{tab:boxqp-convergence-cpu}, we report the average CPU time to terminate with tolerance $\varepsilon=10^{-6}$.
If the algorithm fails to converge in two hours, we report in parentheses the order of the remaining gap in the worst case.
The results confirm that eigenvalue cuts converge fastest to PSD feasibility, although they improve the dual bound slowly during the earliest iterations, as illustrated in Figure~\ref{fig:boxqp-n30-gc-cpu_sel}. The denser the instance, the longer it takes to converge. Dynamically varying the rank of the separator, as in the $\ell^{90\%}_\infty$ and $\ell^{90\%}_2$ strategies, also allows Algorithm~\ref{alg:master_oracle_gauge} to converge for two density values and to reach a good approximation of the PSD solution in the other cases.

\begin{table}[h]
\footnotesize
    \caption{BoxQP instances ($n=30$): average CPU time in seconds 
    when convergence occurs ($\varepsilon=10^{-6}$);
    otherwise, 
    in parentheses, the order of the largest remaining gap 
    $|\lambda_{\min}(X)|$ after two hours.}
    \label{tab:boxqp-convergence-cpu}
    \centering
    \begin{tabular}{rr|r|rrrrrr}
    \multicolumn{2}{l}{} &SDP &\multicolumn{6}{c}{LP}\\
       $d$ &$\#$ &$t_{sdp}$
       &$\ell_1$ &$\ell_\infty$ &$\ell_2$ &$\ell^{90\%}_\infty$ &$\ell^{90\%}_2$ &$\ell^{rand}_\infty$\\
        \hline
       25  &3 &1.4  &16.4 &89.7 &40.8 &18.0 &17.3 &31.0\\
       50  &3 &1.8 &1914 &($10^{-2}$) &($10^{-3}$) &($10^{-5}$) &($10^{-5}$) &($10^{-4}$)\\
       75  &3 &3.0   &1186 &($10^{-3}$) &($10^{-4}$) &2233 &2191 &($10^{-5}$)\\
       100  &3 &3.0  &2491 &($10^{-2}$) &($10^{-3}$) &($10^{-5}$) &($10^{-5}$) &($10^{-3}$)\\
        \end{tabular}
\end{table}

Again, we observe substantial variability in the results across instances with the same nominal size and density: although the shape of the gap closed evolution 
is comparable, the total computation time can differ significantly. For example, for $d=100$, all tested strategies converge in less than $200$ seconds on instance \texttt{030-100-2}, whereas they require more than $6200$ seconds on instance \texttt{030-100-3}. For $d=50$, the time for convergence is less than one second on instance \texttt{030-050-1} considering all strategies, whereas it exceeds more than $5400$ seconds on instance \texttt{030-050-3}.

\paragraph{Large instances ($n=250$)}
For the largest generated BoxQP instances, \texttt{MOSEK} cannot solve the SDP relaxation with the available memory or within a reasonable time, and the SOC relaxation is also too expensive. Hence, Table~\ref{tab:boxqp-n250-gc} reports the gap closed relative to the best bound obtained by any tested cut strategy in the LP setting within 30 minutes. 
Frobenius cuts dominate the short-time (i.e., $60$ seconds) bound improvement for all density values except the smallest. Over longer runs  (i.e., $1800$ seconds),  randomized rank cuts become competitive with Frobenius cuts. 
We select two instances, sparse and dense, and show the lower bound improvement over CPU time in Figure~\ref{fig:boxqp-n250-examples}. Frobenius and nuclear cuts quickly improve the bound, then plateau after about $300$ seconds for the sparse case and $100$ seconds for the dense case, and Frobenius cuts compute better bounds. The eigencut strategy leads to a very slow bound improvement, in particular, for the small instances.

\begin{table}[h]
\footnotesize
    \caption{BoxQP instances ($n=250$): average gap closed $GC^T$ relative to the best bound found within 30 minutes, for $T=60$ seconds (top) and $T=1800$ seconds (bottom).}
    \label{tab:boxqp-n250-gc}
    \centering
    \begin{tabular}{rrr|rrrrr}
    \multicolumn{3}{l}{} &\multicolumn{5}{c}{LP}\\
       $d$ &$\#$ &$T$
       &$\ell_1$ &$\ell_\infty$ &$\ell_2$ &$\ell^{90\%}_2$  &$\ell^{\mathrm{rand}}_\infty$\\
          \hline
       25 &3 &60 &0.0 & \textbf{39.4} &31.1 &0.3 &18.2\\
       50 &3 &60 &0.3 &59.8 &\textbf{62.1} &17.0 &40.6\\
       75 &3 &60 &7.9 &64.7 &\textbf{69.5} &32.8 &44.4\\
       100 &3 &60 &17.8 &70.5 &{\bf 78.4} &44.7 &63.8\\
            \hline
       25 &3 &1800 &5.9 &87.3 &\textbf{100} &75.1 &95.0\\
       50 &3 &1800 &35.6 &93.2 &\textbf{100} &89.3 &97.2\\
       75 &3 &1800 &51.8 &95.3 &\textbf{100} &91.9 &97.6\\
       100 &3 &1800 &58.6 &96.3 &\textbf{100} &93.5 &98.4\\
    \end{tabular}

\end{table}

\begin{figure}
        \centering   \subfigure[]{
        \includegraphics[width=0.48\linewidth]{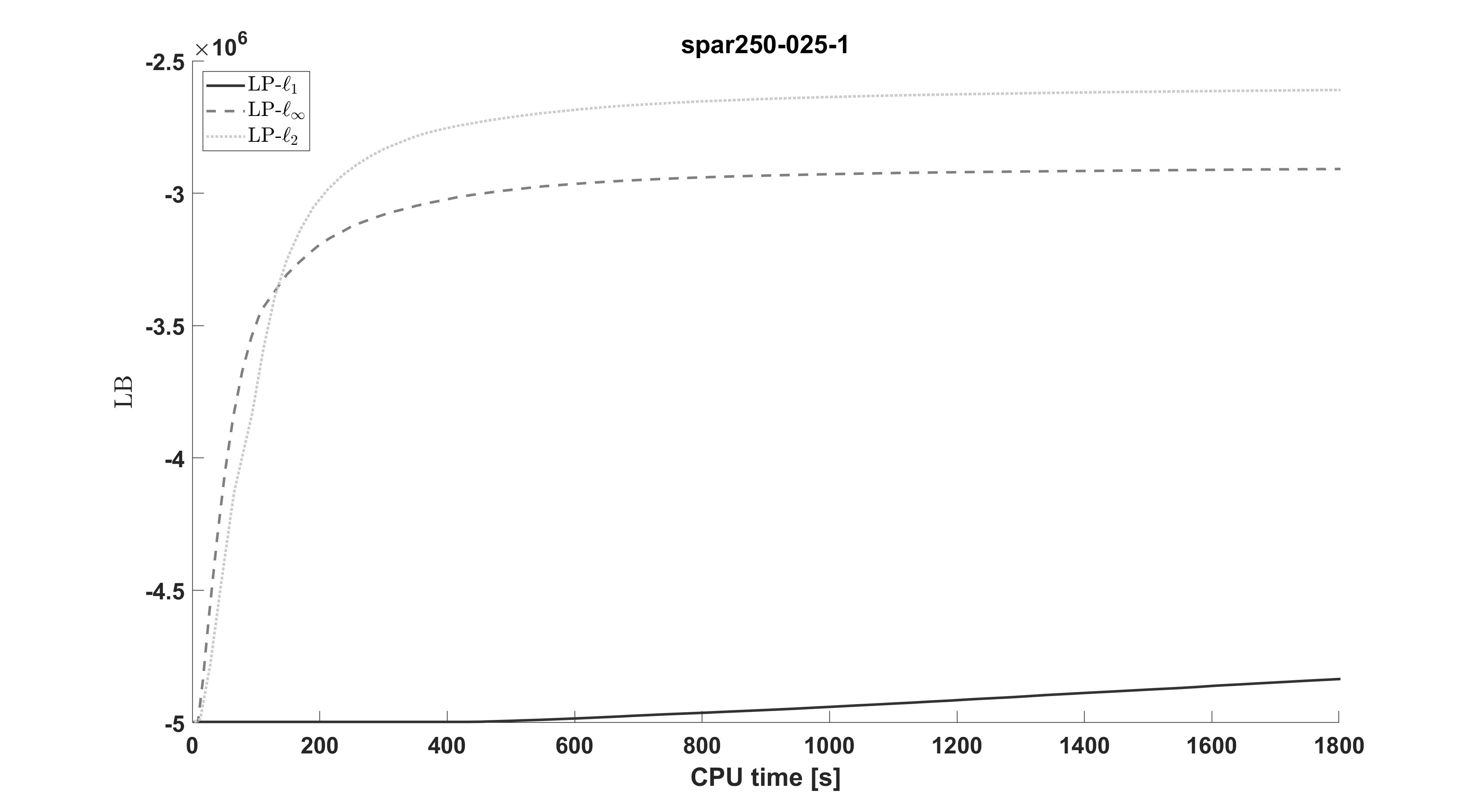}}
        \subfigure[]{
        \includegraphics[width=0.48\linewidth]{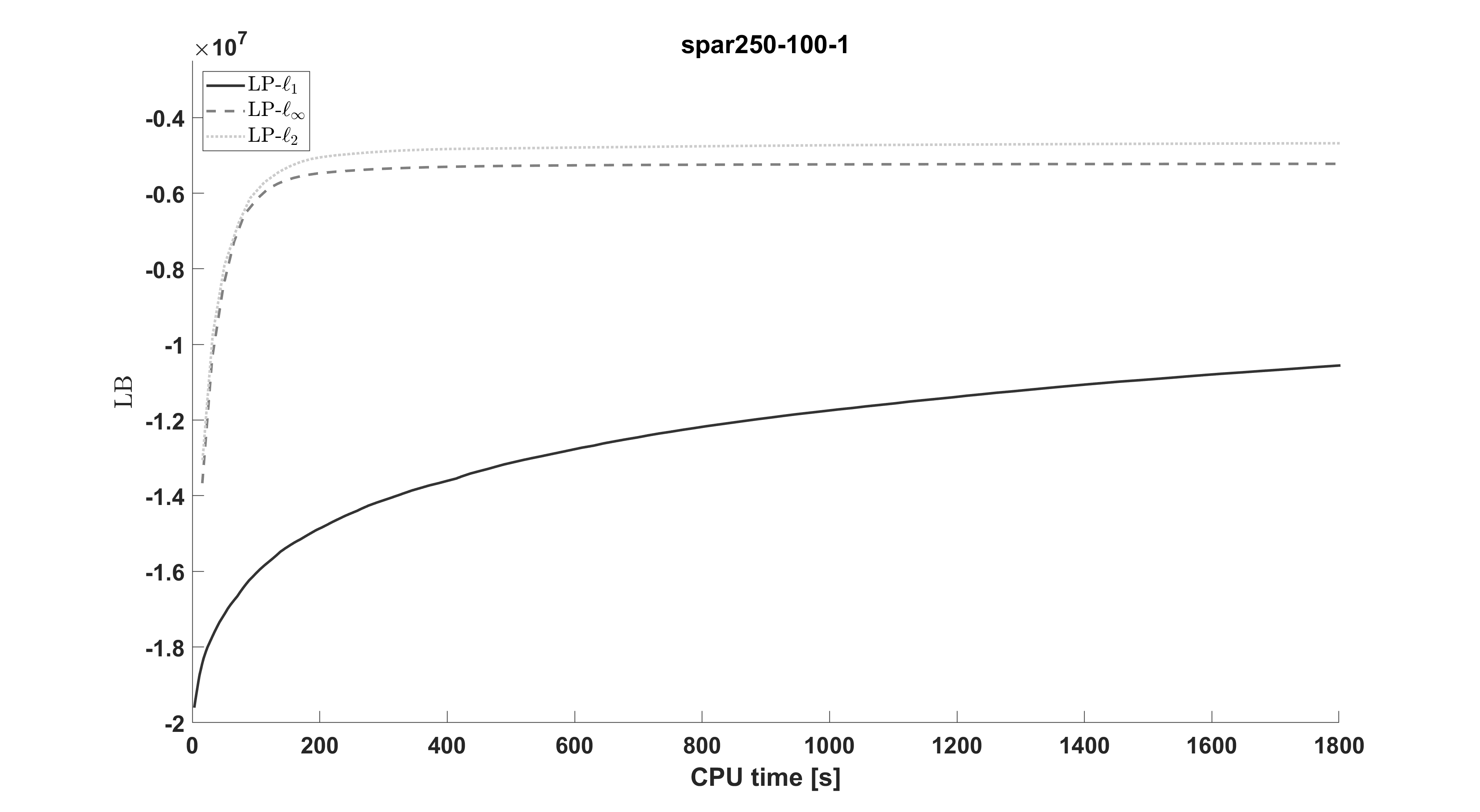}}
    \caption{Lower bound evolution for two BoxQP instances with $n=250$ over $T = 1800$ seconds.}
    \label{fig:boxqp-n250-examples}
\end{figure}

Enforcing a maximum rank bound $m\in[n]$
has no impact on the density of the separators, which remain fully dense in these experiments. However, we observe an effect on the LP solution time immediately after the cut is added. 
As an example, 
when running $\ell^{\mathrm{rand}}_\infty$ for 30 minutes on instance \texttt{250-050-2}: on average, solving the LP takes 6 seconds after adding an eigencut (14 iterations with rank $k=1$) and 37 seconds after adding a nuclear cut (14 iterations with average rank $k=118$).





\subsection{Sparse principal component analysis}
Given a normalized and centered data matrix $A\in\R^{p\times n}$ with $p$ observations in an $n$-dimensional feature space, let $\Sigma:=\frac{1}{p-1}A^\top A\in\Sn$ be the sample covariance matrix. Sparse principal component analysis (SPCA) seeks directions that explain substantial variance while promoting sparsity in the loading vector. The standard formulation is
$$\max x^\top \Sigma x: \|x\|_2 = 1,\: \|x\|_0 \leq t, \: x\in\R^n,$$
where $\|x\|_0$ denotes the number of nonzero entries in $x$. Because of the cardinality constraint $\|x\|_0 \leq t$, this problem is NP-hard, but it admits an SDP relaxation, by setting $X=xx^\top$.
The normalization constraint $\|x\|_2 = 1$ implies $\tr(X)=1$, $X\succeq 0$, and $\operatorname{rank}(X)=1$, 
and $\|x\|_0\le t$ implies $\|X\|_1=\sum_{j,l}|X_{jl}|\le t$
. Dropping the rank constraint 
yields the convex relaxation in~\cite{daspremont_etal_2007}:
\begin{equation}
\min  \ip{-\Sigma}{X}\ :\ 
        \tr(X) = 1,\
        X \succeq 0,\
        \|X\|_1 \leq t,\ {X\in\Sn}.
\end{equation}

The SPCA instances are selected as in~\cite{bertsimas_cory_wright_2020} from the UCI ML repository~\cite{lichman_2013} and we arbitrarily fix density $t=10$. 
For each instance, we first cleaned the data to work only with numerical predictive variables, following the UCI repository documentation.
Features with zero variance led to zero standard deviation and \texttt{NaN} values after standardization. We removed these features before constructing $\Sigma$ (i.e., deleted their rows/columns). This affected the ionosphere, lung, arrhythmia, gait, gastro, and micromass datasets. Reported dimensions reflect this preprocessing.

\paragraph{Small instances ($n<250$)} 
We first focus on the smallest SPCA instances, for which a lower bound can be computed in seconds by solving the SDP relaxation with \texttt{MOSEK}. As before, we consider two settings for the maximum computation time $T$. The top of Table~\ref{tab:spca-small} reports results with $T=t_{sdp}$, while the bottom part reports results obtained with our algorithm running for $60$ seconds. Unlike the BoxQP instances, the SOC relaxation is competitive in short-time settings, particularly when using the Frobenius cuts template. When the time limit is extended to $60$ seconds, our algorithm shows comparable results across the two relaxations and different templates for the two smallest instances ($n=13$). For the other instances, eigencuts and dynamic rank strategy yield the best performance in each of the considered relaxations and, overall, in the SOC relaxation.
\begin{table}[]
\footnotesize
    \caption{SPCA small instances with $t=10$: average gap closed $GC^T$ (\% of the SDP value) for $T=t_{sdp}$ (top) and $T=60$ seconds (bottom).} 
    \label{tab:spca-small}
    \centering
    \begin{tabular}{rrr|rrrr|rrrr}
    \multicolumn{3}{l}{} &\multicolumn{4}{c}{LP} &\multicolumn{3}{c}{SOC}\\
       name &$n$ &$T$ 
       &$\ell_1$ &$\ell_\infty$ &$\ell_2$ &$\ell^{90\%}_2$
       &$\ell_1$ &$\ell_\infty$ &$\ell_2$ &$\ell^{90\%}_2$\\
        \hline
pitprops &13 &0.05 &88.8 &86.7 &89.0 &90.3 &92.8 &83.0 &{\bf 95.2} &92.1\\
wine &13 &0.03 &90.5 &84.6 &89.2 &90.9 &91.3 &83.5 &{\bf 95.1} &91.3\\
ionosphere &34 &0.34 &82.1 &58.5 &76.9 &{\bf 82.4} &48.6 &27.2 &49.6 &48.6\\
lung &54 &1.03 &69.1 &56.2 &65.6 &68.6 &67.4 &49.6 &{\bf 80.2} &67.4\\
geography &68 &2.66 &74.1 &6.7 &66.8 &75.2 &{\bf 87.2} &75.5 &85.1 &84.9\\
communities &101 &32.3 &75.6 &34.8 &57.5 &74.2 &79.2 &28.7 &73.9 &{\bf 79.5}\\
\hline
pitprops &13 &60 &{\bf 100} &{\bf 100} &{\bf 100} &{\bf 100} &{\bf 100} &98.5 &99.8 &{\bf 100}\\
wine &13 &60 &{\bf 100} &{\bf 100} &{\bf 100} &{\bf 100} &{\bf 100} &99.0 &99.9 &{\bf 100}\\
ionosphere &34 &60 &91.2 &69.4 &85.5 &90.8 &{\bf 93.6} &62.3 &90.9 &93.4\\
lung &54 &60 &83.1 &59.9 &77.1 &83.1 &84.6 &52.9 &85.0 &{\bf 86.0}\\
geography &68 &60 &85.3 &19.3 &73.1 &85.0 &{\bf 92.2} &76.2 &88.2 &90.9\\
communities &101 &60 &77.2 &35.1 &59.9 &76.1 &81.3 &29.1 &74.8 &{\bf 81.7}\\ \hline   
\end{tabular}
\end{table}
\begin{figure}
        \centering   \subfigure[]{
        \includegraphics[width=0.48\linewidth]{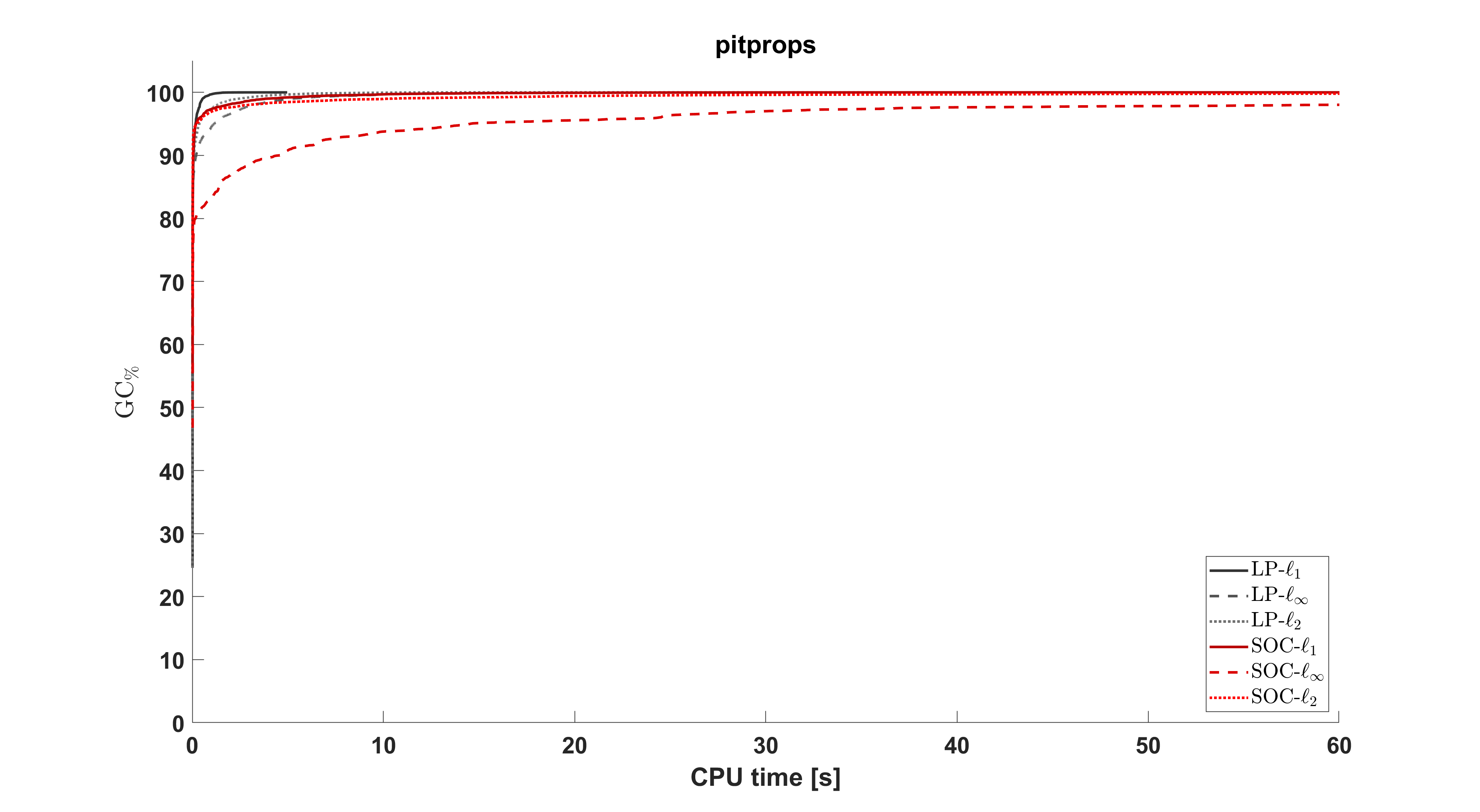}}
        \subfigure[]{
        \includegraphics[width=0.48\linewidth]{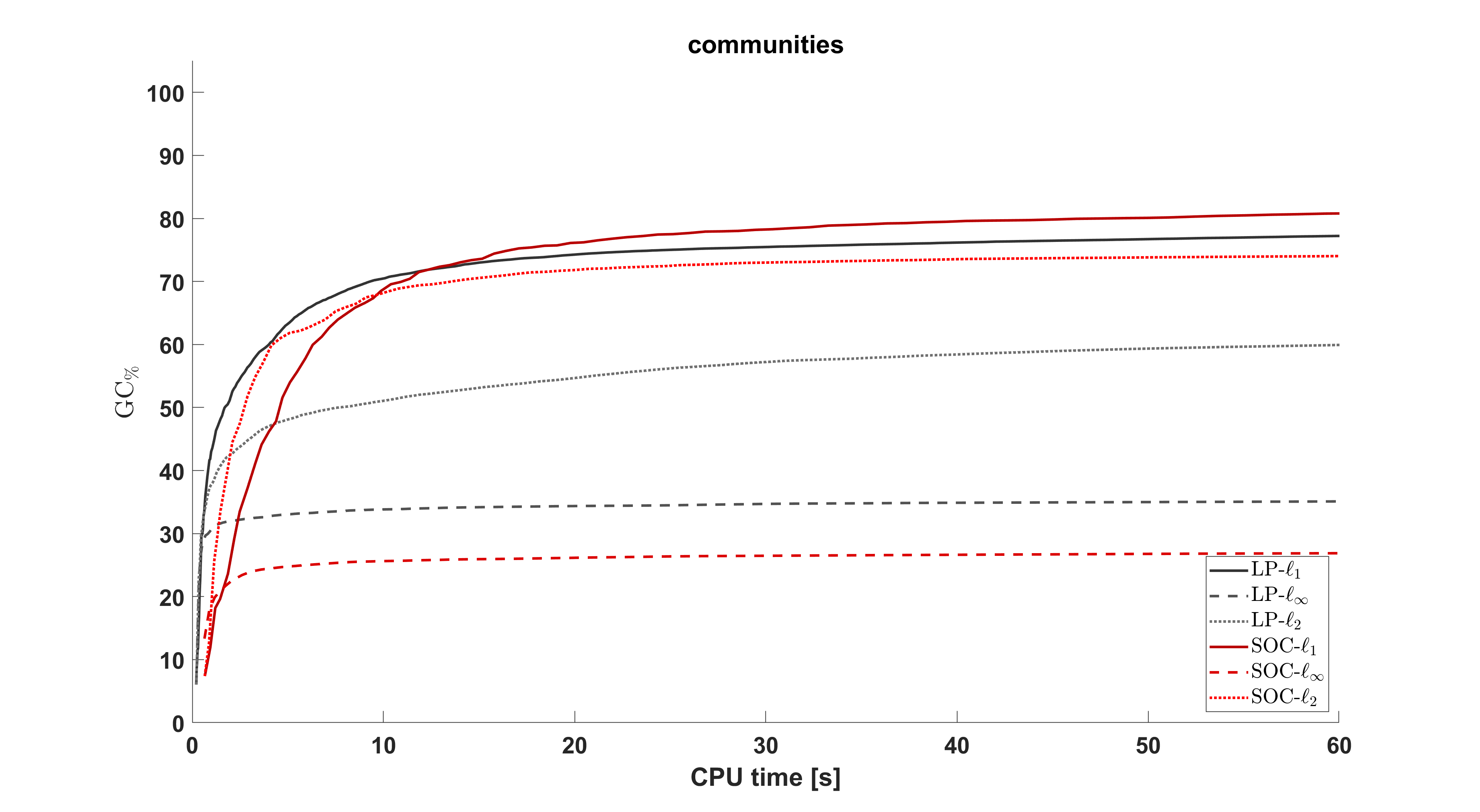}}
    \caption{Gap closed evolution for two small SPCA instances with (a) $n=13$ and (b) $n=101$.} 
    \label{fig:small-SPCA-examples}
\end{figure}
The evolution of the gap-closed is given in Figure~\ref{fig:small-SPCA-examples} for the instances \texttt{pitprops} ($n=13$) and \texttt{communities} ($n=101$). For the first case, we can see that once the relaxation is fixed, the eigencuts strategy closes the gap the fastest, followed by the Frobenius cuts and finally the nuclear cuts.  We mention that in the first $0.15$ seconds, the SOC relaxation with either the Frobenius cuts or the eigencuts has the highest closed gap. After that, this value is overcome by combining the LP relaxation with the eigencuts, which allows the gap to be closed in less than $10$ seconds. This behavior is reversed for the other instance, where the LP relaxation with eigencuts dominates in the first $10$ seconds, while the SOC relaxation allows better results in the long run. However, once the relaxation is fixed, the performance of the different cut strategies is the same as in the smallest case.


\paragraph{Large instances ($n\geq 250$)} 
On larger SPCA instances, \texttt{MOSEK} is unable to find a solution due to memory issues, so we compute the gap closed relative to the best value achieved within $30$ minutes among all cut strategies. In the upper part of Table~\ref{tab:spca-large}, we report the results of running Algorithm~\ref{alg:master_oracle_gauge} with a time limit set to $60$ seconds.  
The SOC relaxation yields the best early strategy for the smallest instance with Frobenius cuts, but its solution time exceeds 1 minute on the three largest instances. 
For all instances but two, the LP relaxation with eigencuts provides the largest gap closed. 
Even with a 1800-second time limit, SOC relaxation is not recommended for the largest instances, where LP relaxation with eigencut or dynamic rank achieves high gap-closed values. Figure~\ref{fig:large-SPCA-examples} shows the evolution of the lower bound for two selected instances, underscoring once again that the best combination, that is, relaxation and cut strategy, is instance dependent. Indeed, with $T=1800$ seconds, SOC relaxation with eigencuts yields the best lower bound for \texttt{arrhythmia} but is the worst strategy on the largest instance \texttt{micromass}.



\begin{table}[h]
\footnotesize
    \caption{SPCA large instances with $t=10$: average gap closed $GC^T$ (\% of the best value) for $T=60$ seconds (top) and $T=1800$ seconds (bottom).}
    \label{tab:spca-large}
    \centering
    \begin{tabular}{rrr|rrrr|rrrr}
    \multicolumn{3}{l}{} &\multicolumn{4}{c}{LP} &\multicolumn{4}{c}{SOC}\\
       name &$n$ &$T$ 
       &$\ell_1$ &$\ell_{\infty}$ &$\ell_2$  &$\ell_2^{90\%}$
       &$\ell_1$ &$\ell_{\infty}$ &$\ell_2$ &$\ell_2^{90\%}$\\
        \hline
arrhythmia &257 &60 &76.8 &53.7 &65.5 &76.3 &65.3 &44.0 &{\bf 80.7} &65.3\\
voice &310 &60 &{\bf 27.2} &3.9 &8.6 &24.7 &3.2 &1.3 &4.0 &3.2\\
gait &320 &60 &{\bf 62.3} &22.0 &52.3 &61.2 &32.5 &14.9 &39.3 &33.3\\
gastro &466 &60 &{\bf 12.5} &0.5 &3.0 &12.4 &- &- &- &-\\
parkinson &754 &60 &{\bf 15.0} &1.8 &10.8 &{\bf 15.0} &- &- &- &-\\
micromass &1139 &60 &19.1 &{\bf 26.1} &23.7 &19.2 &- &- &- &-\\
\hline
arrhythmia &257 &1800 &95.8 &57.5 &77.3 &96.3 &99.6 &48.5 &95.7 &{\bf 100}\\
voice &310 &1800 &98.1 &4.4 &26.3 &{\bf 100} &81.1 &3.4 &48.3 &96.2\\
gait &320 &1800 &92.1 &29.7 &63.3 &91.4 &99.0 &24.7 &86.7 &{\bf 100}\\
gastro &466 &1800 &99.9 &0.6 &21.0 &{\bf 100} &38.6 &0.5 &0.7 &63.8\\
parkinson &754 &1800 &98.6 &2.8 &39.4 &{\bf 100} &17.9 &2.3 &18.8 &21.2\\
micromass &1139 &1800 &{\bf 100} &67.0 &79.7 &98.1 &27.3 &60.7 &35.3 &27.3\\
\hline
\end{tabular}
\end{table}

\begin{figure}
        \centering   \subfigure[]{
        \includegraphics[width=0.48\linewidth]{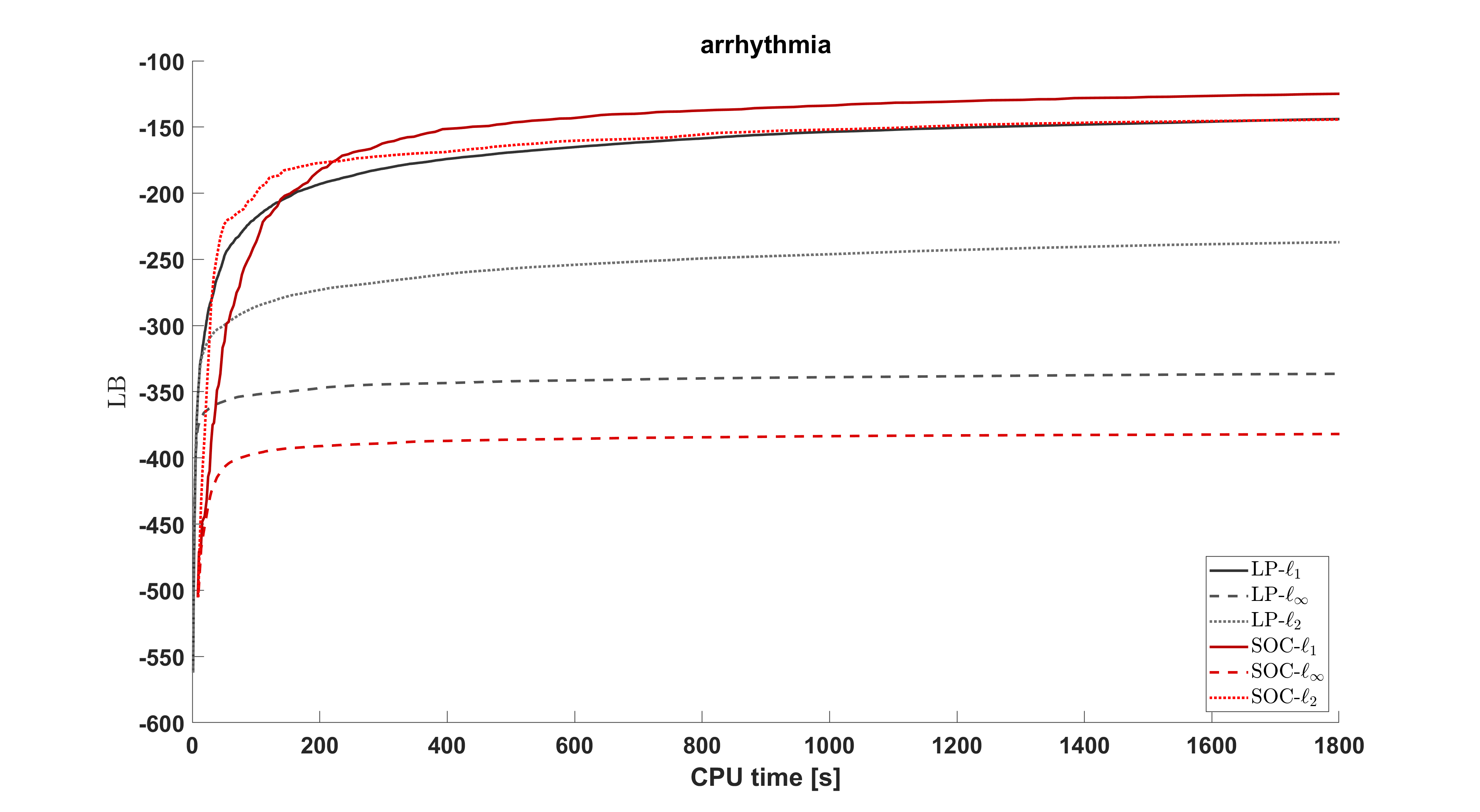}}
        \subfigure[]{
        \includegraphics[width=0.48\linewidth]{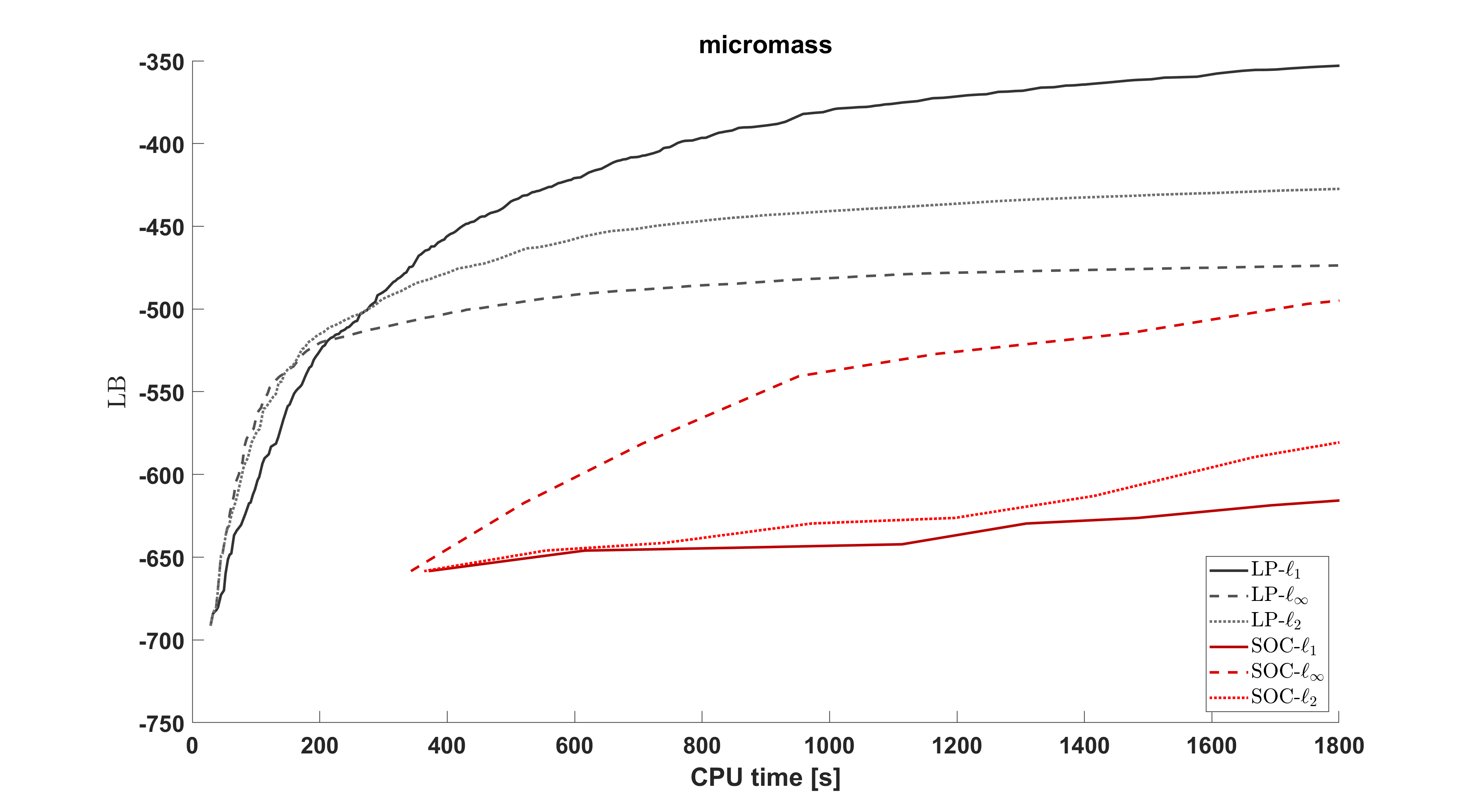}}
    \caption{Lower bound evolution for two large SPCA instances with (a) $n=257$ and (b) $n=1139$.} 
    \label{fig:large-SPCA-examples}
\end{figure}

\section{Conclusions and future work} We introduce spectral-gauge cuts, a canonical family of cutting planes for positive semidefinite (PSD) programming that leverage spectral information through gauge functions. 
The proposed framework unifies and extends existing outer-approximations of the PSD cone, taking sparsity and rank constraints into account. On the theoretical side, we characterized optimal separators, and we analyzed how the choice of gauge and normalization affects the cut strength. 

Numerical experiments on box-constrained quadratic programs and sparse principal component analysis show 
that, in general, combining an LP relaxation with the classical eigencuts or the proposed Frobenius cuts provides a robust trade-off between bound quality and computational burden. This is even more evident in large instances where interior-point methods encounter memory limitations, our 
algorithm remains applicable and produces good-quality bounds.

Theory and practice indicate that no spectral-gauge cut template is strictly dominated by the others. Hence, future work could be dedicated to designing adaptive gauge-selection and dynamic-rank strategies. 

\bibliographystyle{siamplain}
\bibliography{references}

\begin{thebibliography}{10}

\bibitem{ahmadi_majumdar_2019}
{\sc A.~A. Ahmadi and A.~Majumdar}, {\em {DSOS} and {SDSOS} optimization: More
  tractable alternatives to sum of squares and semidefinite optimization}, SIAM
  J. Appl. Algebra Geom., 3 (2019), pp.~193--230.

\bibitem{andersen_etal_2011}
{\sc M.~S. Andersen, J.~Dahl, Z.~Liu, and L.~Vandenberghe}, {\em Interior-point
  methods for large-scale cone programming}, in Optimization for Machine
  Learning, S.~Sra, S.~Nowozin, and S.~J. Wright, eds., MIT Press, 2011,
  pp.~55--83.

\bibitem{baltean_lugojan_etal_2019}
{\sc R.~Baltean-Lugojan, P.~Bonami, R.~Misener, and A.~Tramontani}, {\em
  Scoring positive semidefinite cutting planes for quadratic optimization via
  trained neural networks}.
\newblock Optimization Online preprint, 2019.

\bibitem{battista_de_santis_2025}
{\sc F.~Battista and M.~De~Santis}, {\em Dealing with inequality constraints in
  large-scale semidefinite relaxations for graph coloring and maximum clique
  problems}, {4OR}, 23 (2025), pp.~65--95.

\bibitem{bellavia_gondzio_porcelli_2021}
{\sc S.~Bellavia, J.~Gondzio, and M.~Porcelli}, {\em A relaxed interior-point
  method for low-rank semidefinite programming problems with applications to
  matrix completion}, J. Sci. Comput., 89 (2021).

\bibitem{berk_bertsimas_2019}
{\sc L.~Berk and D.~Bertsimas}, {\em Certifiably optimal sparse principal
  component analysis}, Math. Program. Comput., 11 (2019), pp.~381--420.

\bibitem{bertsimas_cory_wright_2020}
{\sc D.~Bertsimas and R.~Cory-Wright}, {\em On polyhedral and second-order cone
  decompositions of semidefinite optimization problems}, Oper. Res. Lett., 48
  (2020), pp.~78--85.

\bibitem{blekherman_parrilo_thomas_2012}
{\sc G.~Blekherman, P.~A. Parrilo, and R.~R. Thomas}, eds., {\em Semidefinite
  optimization and convex algebraic geometry}, vol.~13 of MOS-SIAM Series on
  Optimization, SIAM, 2012.

\bibitem{boman_etal_2005}
{\sc E.~G. Boman, D.~Chen, O.~Parekh, and S.~Toledo}, {\em On factor width and
  symmetric {H}-matrices}, Linear Algebra Appl., 405 (2005), pp.~239--248.

\bibitem{burer_vandenbussche_2009}
{\sc S.~Burer and D.~Vandenbussche}, {\em Globally solving box-constrained
  nonconvex quadratic programs with semidefinite-based finite
  branch-and-bound}, Comput. Optim. Appl., 43 (2009), pp.~181--195.

\bibitem{chandrasekaran_etal_2012}
{\sc V.~Chandrasekaran, B.~Recht, P.~A. Parrilo, and A.~S. Willsky}, {\em The
  convex geometry of linear inverse problems}, Found. Comput. Math., 12 (2012),
  pp.~805--849.

\bibitem{chen_etal_2025}
{\sc S.~Chen, Y.-J. Liu, J.~Yu, and W.~Zhou}, {\em A semismooth {Newton}-based
  augmented {Lagrangian} algorithm for {Lov{\'a}sz} theta {SDP} problem},
  Optimization,  (2025), pp.~1--22.

\bibitem{daniilidis_etal_2013}
{\sc A.~Daniilidis, D.~Drusvyatskiy, and A.~S. Lewis}, {\em Orthogonal
  invariance and identifiability}, SIAM J. on Matrix Anal. and Appl., 35
  (2014), pp.~580--598.

\bibitem{daspremont_etal_2007}
{\sc A.~d'Aspremont, L.~El~Ghaoui, M.~I. Jordan, and G.~R. Lanckriet}, {\em A
  direct formulation for sparse pca using semidefinite programming}, SIAM Rev.,
  49 (2007), pp.~434--448.

\bibitem{dey_etal_2022}
{\sc S.~S. Dey, A.~M. Kazachkov, A.~Lodi, and G.~Mu{\~n}oz}, {\em Cutting plane
  generation through sparse principal component analysis}, SIAM J. Optim., 32
  (2022), pp.~1319--1343.

\bibitem{dey_molinaro_2018}
{\sc S.~S. Dey and M.~Molinaro}, {\em Theoretical challenges towards
  cutting-plane selection}, Math. Program., 170 (2018), pp.~237--266.

\bibitem{gunluk_etal_2026}
{\sc O.~G{\"u}nl{\"u}k, P.~J{\"u}nger, J.~Linderoth, A.~Lodi, and J.~Luedtke},
  {\em Sparse cuts for the positive semidefinite cone}.
\newblock arXiv:2603.09864 preprint, 2026.

\bibitem{horn_johnson_2013}
{\sc R.~A. Horn and C.~R. Johnson}, {\em Matrix analysis}, Cambridge University
  Press, 2~ed., 2013.

\bibitem{kelley_1960}
{\sc J.~E. Kelley, Jr.}, {\em The cutting-plane method for solving convex
  programs}, J. Soc. Ind. Appl. Math., 8 (1960), pp.~703--712.

\bibitem{lichman_2013}
{\sc M.~Kelly, R.~Longjohn, and K.~Nottingham}, {\em The {UCI} machine learning
  repository}.
\newblock https://archive.ics.uci.edu.

\bibitem{krishnan_mitchell_2006}
{\sc K.~Krishnan and J.~E. Mitchell}, {\em A unifying framework for several
  cutting plane methods for semidefinite programming}, Optim. Methods Softw.,
  21 (2006), pp.~57--74.

\bibitem{locatelli_etal_2025}
{\sc M.~Locatelli, V.~Piccialli, and A.~M. Sudoso}, {\em Fix and bound: an
  efficient approach for solving large-scale quadratic programming problems
  with box constraints}, Math. Program. Comput., 17 (2025), pp.~231--263.

\bibitem{malick_etal_2009}
{\sc J.~Malick, J.~Povh, F.~Rendl, and A.~Wiegele}, {\em Regularization methods
  for semidefinite programming}, SIAM J. Optim., 20 (2009), pp.~336--356.

\bibitem{nesterov_nemirovskii_1994}
{\sc Y.~Nesterov and A.~Nemirovskii}, {\em Interior-point polynomial algorithms
  in convex programming}, vol.~13 of SIAM Studies in Applied Mathematics, SIAM,
  1994.

\bibitem{permenter_parrilo_2018}
{\sc F.~Permenter and P.~A. Parrilo}, {\em Partial facial reduction:
  Simplified, equivalent {SDP}s via approximations of the {PSD} cone}, Math.
  Program., 171 (2018), pp.~1--54.

\bibitem{povh_rendl_wiegele_2006}
{\sc J.~Povh, F.~Rendl, and A.~Wiegele}, {\em A boundary point method to solve
  semidefinite programs}, Computing, 78 (2006), pp.~277--286.

\bibitem{qualizza_belotti_margot_2012}
{\sc A.~Qualizza, P.~Belotti, and F.~Margot}, {\em Linear programming
  relaxations of quadratically constrained quadratic programs}, in Mixed
  Integer Nonlinear Programming, J.~Lee and S.~Leyffer, eds., vol.~154 of IMA
  Vol. in Math. and Appl., Springer, 2012, pp.~407--426.

\bibitem{ramana_1993}
{\sc M.~V. Ramana}, {\em An algorithmic analysis of multiquadratic and
  semidefinite programming problems}, PhD thesis, Johns Hopkins University,
  Baltimore, MD, 1993.

\bibitem{ramana_goldman_1995}
{\sc M.~V. Ramana and A.~J. Goldman}, {\em Some geometric results in
  semidefinite programming}, J. Glob. Optim., 7 (1995), pp.~33--50.

\bibitem{rudin_1991}
{\sc W.~Rudin}, {\em Functional Analysis}, McGraw--Hill, 2~ed., 1991.

\bibitem{sherali_dalkiran_desai_2012}
{\sc H.~D. Sherali, E.~Dalkiran, and J.~Desai}, {\em Enhancing {RLT}-based
  relaxations for polynomial programming problems via a new class of
  {$v$}-semidefinite cuts}, Comput. Optim. Appl., 52 (2012), pp.~483--506.

\bibitem{sherali_fraticelli_2002}
{\sc H.~D. Sherali and B.~M.~P. Fraticelli}, {\em Enhancing {RLT} relaxations
  via a new class of semidefinite cuts}, J. Glob. Optim., 22 (2002),
  pp.~233--261.

\bibitem{shor_1987}
{\sc N.~Z. Shor}, {\em Quadratic optimization problems}, Sov. J. Comput. Syst.
  Sci., 25 (1987), pp.~1--11.
\newblock Translated from \emph{Tekhnicheskaya Kibernetika}, No.\ 1 (1987),
  pp.\ 128--139.

\bibitem{shor_1998}
{\sc N.~Z. Shor}, {\em Nondifferentiable optimization and polynomial problems},
  vol.~24 of Nonconvex Optimization and Its Applications, Springer, Boston, MA,
  1998.

\bibitem{sun_etal_2020}
{\sc D.~Sun, K.-C. Toh, Y.~Yuan, and X.-Y. Zhao}, {\em {SDPNAL}+: {A} {Matlab}
  software for semidefinite programming with bound constraints}, Optim. Methods
  Softw., 35 (2020), pp.~87--115.

\bibitem{vandenberghe_boyd_1996}
{\sc L.~Vandenberghe and S.~Boyd}, {\em Semidefinite programming}, SIAM Rev.,
  38 (1996), pp.~49--95.

\bibitem{vandenbussche_nemhauser_2005}
{\sc D.~Vandenbussche and G.~Nemhauser}, {\em A branch-and-cut algorithm for
  nonconvex quadratic programs with box constraints}, Math. Program., 102
  (2005), pp.~559--575.

\bibitem{wang_tanaka_yoshise_2021}
{\sc Y.~Wang, A.~Tanaka, and A.~Yoshise}, {\em Polyhedral approximations of the
  semidefinite cone and their application}, Comput. Optim. Appl., 78 (2021),
  pp.~893--913.

\bibitem{wesselmann_suhl_2012}
{\sc F.~Wesselmann and U.~H. Suhl}, {\em Implementing cutting plane management
  and selection techniques}, tech. report, University of Paderborn, 2012.

\bibitem{yildirim_2026}
{\sc E.~A. Y{\i}ld{\i}r{\i}m}, {\em Relaxations of {KKT} conditions do not
  strengthen finite {RLT} and {SDP-RLT} bounds for nonconvex quadratic
  programs}, J. Glob. Optim., 94 (2026), pp.~891--918.

\end{thebibliography}

\medskip

\appendix
\section{Proof of Proposition~\ref{prop:sparsified_basic}}\label{app:proof_phi} We recall the two notions used below. A set $C\subseteq\mathbb{R}^n$ is \textit{balanced} if $\alpha C\subseteq C$ for every scalar $\alpha$ with $|\alpha|\le1$, and it is \textit{absorbing} if, for every $\mathbf{z}\in\mathbb{R}^n$, there exists $t>0$ such that $\mathbf{z}\in tC$.

\begin{proof}[Proof of Proposition~\ref{prop:sparsified_basic}]The set $\mathcal{A}_k(\phi)$ is compact since it is the intersection of the compact $\phi$-unit ball with the finite union of coordinate subspaces corresponding to supports of size at most $k$. Hence $\mathcal{K}_k(\phi)=\operatorname{conv}(\mathcal{A}_k(\phi))$ is compact and convex. Moreover, $\mathcal{A}_k(\phi)$ is balanced. Indeed, if $\mathbf{u}\in\mathcal{A}_k(\phi)$ and $|\alpha|\le1$, then
\[
\phi(\alpha\mathbf{u})=|\alpha|\phi(\mathbf{u})\le1,
\quad
\|\alpha\mathbf{u}\|_0\le\|\mathbf{u}\|_0\le k.
\]
Therefore, $\mathcal{K}_k(\phi)$ is balanced as well.

Let $c:=\phi(\mathbf{e}_1)>0$. Since $\phi$ is a symmetric gauge, $\phi(\mathbf{e}_i)=c$ for every $i\in[n]$, and hence $\pm\mathbf{e}_i/c\in\mathcal{A}_k(\phi)$ for every $i\in[n]$. Thus
\[
\operatorname{conv}\Big\{\pm\frac{\mathbf{e}_i}{c}:\ i\in[n]\Big\}
\subseteq
\mathcal{K}_k(\phi).
\]
The set on the left is a full-dimensional cross-polytope, so $\mathcal{K}_k(\phi)$ is absorbing.

The function $\phi^{\langle k\rangle}$ in~\eqref{eq:phi_m_def} is therefore the gauge function of a compact, convex, balanced, absorbing set. It is consequently a seminorm~\cite[Thm.~1.35]{rudin_1991}.

To see that it is positive definite, note that
\[
\mathcal{A}_k(\phi)\subseteq\{\mathbf{u}:\phi(\mathbf{u})\le1\}.
\]
Since the $\phi$-unit ball is convex, $\mathcal{K}_k(\phi)\subseteq\{\mathbf{u}:\phi(\mathbf{u})\le1\}$. Hence, whenever $\mathbf{z}\in t\mathcal{K}_k(\phi)$, we have $\phi(\mathbf{z})\le t$. Taking the infimum over such $t>0$ yields
\[
\phi(\mathbf{z})\le\phi^{\langle k\rangle}(\mathbf{z}).
\]
Thus $\phi^{\langle k\rangle}(\mathbf{z})>0$ for every $\mathbf{z}\neq\mathbf{0}$, and $\phi^{\langle k\rangle}$ is a norm.

Finally, $\mathcal{A}_k(\phi)$ is invariant under sign changes and permutations, because both $\phi$ and $\|\cdot\|_0$ are invariant under these operations. The same is true of $\mathcal{K}_k(\phi)$, and therefore of its gauge. Hence $\phi^{\langle k\rangle}$ is a symmetric gauge.
\end{proof}

\end{document}